\documentclass[11pt, reqno]{amsart}

\usepackage{enumerate}
\usepackage{graphicx,graphics}
\usepackage{amsfonts, amssymb, amsthm, amsmath}
\usepackage{mathrsfs,color}
\usepackage{listings}
\usepackage{orcidlink}
\usepackage{tikz-cd}
\input{xy}
\xyoption{all}
\usepackage{tcolorbox}

\usepackage[style=numeric,sorting=nyt,giveninits=true,maxnames=99,url=false]{biblatex} 
\DeclareNameAlias{sortname}{family-given}
\DeclareNameAlias{default}{family-given}

\newtheorem{theorem}{Theorem}[section]
\newtheorem{lemma}[theorem]{Lemma}

\newtheorem{remark}[theorem]{Remark}
\newtheorem{corollary}[theorem]{Corollary}

\newtheorem{proposition}[theorem]{Proposition}
\newtheorem{example}[theorem]{Example}
\newtheorem{definition}[theorem]{Definition}

\begin{document}

\title{Analytic integration of metric-valued functions in Lipschitz free spaces}

\author[R. Arnau]{Roger Arnau \orcidlink{0000-0003-2544-8875}}
\address[Roger Arnau]{Instituto Universitario de Matem\'atica Pura y Aplicada\\
Universitat Polit\`ecnica de Val\`encia\\
Camino de Vera s/n, 46022 Valencia, Spain}
\email{ararnnot@upvnet.es}

\author[\'A. Gons\'alez Cort\'es]{\'Alvaro Gonz\'alez Cort\'es$^*$ \orcidlink{0009-0001-2328-2173}}
\address[\'Alvaro Gonz\'alez Cort\'es]{Instituto Universitario de Matem\'atica Pura y Aplicada\\
Universitat Polit\`ecnica de Val\`encia\\
Camino de Vera s/n, 46022 Valencia, Spain}
\email{agoncor@upv.es}

\author[E. A. S\'anchez P\'erez]{Enrique A. S\'{a}nchez P\'{e}rez \orcidlink{0000-0001-8854-3154}}
\address[Enrique A. S\'{a}nchez P\'{e}rez]{Instituto Universitario de Matem\'atica Pura y Aplicada\\
Universitat Polit\`ecnica de Val\`encia\\
Camino de Vera s/n, 46022 Valencia, Spain}
\email{easancpe@mat.upv.es}

\thanks{$^*$ Corresponding author}

\begin{abstract}

    We develop an integration theory for functions taking values in a
    metric space. Following a Bochner-type construction, we define the concept of free integral as an element of the Lipschitz-free space $\mathcal{F}(M)$. We establish the main properties of this integral, including duality formulas, and the study of the resulting space of free integrable functions. We also cover when the metric space is a Banach space: in this setting, the free integral has an interpretable decomposition generalising the Bochner integral. We then connect the free integral with the geometry of $\mathcal{F}(M)$ by showing that it always produces convex integrals of molecules. This allows to study extremal properties within the unit ball of $\mathcal{F}(M)$. Finally, we provide a detailed example to illustrate the framework we develop.
    
\end{abstract}

\maketitle

\section{Introduction}
Integration of vector-valued functions with respect to scalar measures has become a fundamental tool in modern Functional Analysis. The classical theories of Bochner and Pettis integration found important applications from the outset, but their decisive impact emerged when they were deeply connected to the geometry of Banach spaces through the Radon--Nikod\'ym theory in the 1970s. A systematic account of this interplay between vector measures, integration, and Banach space geometry can be found in \cite{diestelmeasures}. In parallel, the dual
framework (namely, integration of scalar-valued functions with respect to vector measures) proved to be especially fruitful in the theory of operators on Banach function lattices. In this setting, function spaces, vector measures, integrals, and operators can be understood as different manifestations of closely related mathematical structures; see \cite{orsp} for a comprehensive treatment.

Currently, a major research direction in analysis consists of transferring classical linear tools into the nonlinear setting. Metric analysis, strongly connected to Functional Analysis through linearisation techniques for Lipschitz mappings and the theory of Lipschitz-free spaces (also known as Arens--Eells spaces), has attracted sustained attention in recent years. Given a pointed metric space $(M,d,\theta)$, the Lipschitz-free space $\mathcal{F}(M)$ is a
Banach space that contains an isometric copy of $M$ and satisfies a universal property: every Lipschitz map from $M$ into a Banach space extends uniquely to a bounded linear operator on $\mathcal{F}(M)$. This mechanism allows one to embed nonlinear structures into linear ones, preserving enough geometry to apply functional-analytic methods. We refer to \cite{weaverlipschitz} for a comprehensive study of Lipschitz-free spaces and to \cite{godefroy2015,godefroylipschitz} for surveys on their applications.

The problem of extending the notion of integration to functions taking values in a general metric space has been approached from several perspectives. For example, geometric concepts of integrals in metric spaces have been developed to determine the centre of mass of a probability distribution. A classical line of work defines the mean of a metric-valued random variable as the minimiser of the expected squared distance: given a function $f\colon\Omega\to M$ on a probability space, the
\emph{Fr\'echet mean} is
\begin{equation}\label{eq:frechet}
    \bar{x} = \operatorname*{argmin}_{y\in M}
    \int_\Omega d^2(f(\omega),y)\,d\mu(\omega).
\end{equation}
This approach is rooted in the work of
Fr\'echet \cite{frechet1948elements}, and it has been developed by Sturm
\cite{sturm2003probability} in spaces of nonpositive curvature and further
studied in various geometric settings (see for example \cite{ohta2012barycenters}). 

However, the aim of this paper is to develop an integral for metric-valued functions from an analytical rather than geometric perspective, producing an object that captures the full distributional information and not merely its centre of mass. Our approach is to define the integral in the associated free space: for a metric-valued function $f\colon\Omega\to M$, we
construct its \emph{free integral} as an element of $\mathcal{F}(M)$. This is achieved by first integrating simple functions and then extending by approximation, guided by the Bochner integration theory in $\mathcal{F}(M)$. A central feature of this construction is that, when $M$ is a Banach space $X$, the free integral carries strictly more information than the Bochner integral, providing a meaningful generalisation.

The paper is organised as follows. Section \ref{sec:prelim} introduces the necessary background on measure spaces, Lipschitz functions, and Lipschitz-free spaces. 

In Section \ref{sec:freeint}, we develop the free integral following a Lebesgue-type approach: after adapting some relevant notions of measurability, we first define the free integral for simple functions and then extend it to pointwise limits of simple ones. We establish its main properties, including the identification with the Bochner integral, duality formulas, and functoriality under Lipschitz maps. We then study the Banach-valued case, revealing the additional information captured by the free integral beyond the Bochner integral. The section concludes with the study of the metric space of free integrable functions and the linearisation of the free integration operator.

Section \ref{sec:geometry} connects the free integral with the geometry of $\mathcal{F}(M)$. In particular, we link the free integral with the theory of convex integrals of molecules recently introduced in \cite{aliagaconvex}. We prove that the free integral provides a systematic source of convex integrals of molecules, connecting the integration theory with the extremal geometry of free spaces.

Finally, in Section \ref{sec:example}, we present a step-by-step example to illustrate the theoretical framework established throughout the paper. By explicitly describing the Lipschitz-free space of a finite metric tree and computing some free integrals, we show how these mathematical structures work in practice.

\section{Preliminaries}\label{sec:prelim}
In this section we present some background concepts needed for the rest of the paper. We begin with basic notions from measure theory, and then recall the theory of Lipschitz functions and Lipschitz-free spaces.

\subsection{Measure theory and the Bochner integral}
 Let $\Omega$ be a non-empty set, and let $\Sigma$ be a $\sigma$-algebra over $\Omega$. A finite measure is a function $\mu\colon\Sigma\to\mathbb R$ such that \begin{enumerate}[(M1)]
    \item $\mu(A)\geq 0$ for all $A\in\Sigma,$
    \item For every countable family $\{A_k\}_{k=1}^\infty\subseteq \Sigma$ of pairwise disjoint sets, it holds that $$\mu\left(\displaystyle\bigcup_{k=1}^\infty A_k\right)=\displaystyle\sum_{k=1}^\infty \mu(A_k).$$
\end{enumerate}
In this case, it is said that $(\Omega, \Sigma,\mu)$ is a \emph{finite measure space}, and every $A\in\Sigma$ is said to be \emph{measurable}. We emphasize that $\mu$ takes values in $\mathbb R$, so all measurable sets have finite measure. Throughout this paper, we fix such a measure space. A set $N\in\Sigma$ is called a \emph{null set} if $\mu(N)=0$. A pointwise property involving $\omega\in\Omega$ is said to hold \emph{almost everywhere} (abbreviated as $\mu$-a.e.) if the set of points where it fails is null. When only (M2) holds, $\mu$ is said to be a signed measure.



Since the inception of the theory of integration for scalar-valued functions (Riemann, Lebesgue), numerous theoretical frameworks have been proposed to generalize the notion of the integral to functions taking values in general vector spaces. One important such generalization is the Bochner integral, which we briefly present below. For a comprehensive treatment and more details, we refer the reader to \cite{aliprantis2006}.

Let $(X,\|\cdot\|_X$) be a Banach space. A function $S\colon\Omega\to X$ is called simple if assumes only a finite number of values $x_1,\ldots,x_n$ and $S^{-1}(\{x_i\})\in\Sigma$ for each $i$. In this case, $S$ can be represented as $S=\sum_{i=1}^n x_i\chi_{A_i}$. The integral of such a simple function is the vector $\int Sd\mu$ defined by $$\int Sd\mu=\sum_{i=1}^n \mu(A_i)x_i.$$
It is important to note that if $S=\sum_{j=1}^my_j\chi_{B_j}$ is another representation for $S$, then
$$\int Sd\mu=\sum_{j=1}^m\mu(B_j)y_j.$$

It is said that $f\colon\Omega\to X$ is strongly $\mu$-measurable (or strongly measurable if there is no confusion regarding the measure) if there exists a sequence of simple functions $\{S_k\}_k$ pointwise convergent to $f$ $\mu$-a.e.

Let $f\colon\Omega\to X$ be a strongly measurable function. Then, $f$ is Bochner integrable if there exist a sequence of simple functions $\{S_k\}_k$ such that $\|f-S_k\|_X$ is Lebesgue integrable for each $k$, and
$$\lim_n\int  \|f-S_k\|_Xd\mu=0.$$
In this case, the Bochner integral of $f$ is defined as
$$\int fd\mu=\lim_k \int S_kd\mu,$$
where the latter limit is taken with respect to the topology induced by the norm on $X$.

The Bochner integral is well defined, in the sense that its value is independent of the particular sequence of simple functions chosen to approximate $f$. Moreover, the set of all Bochner integrable functions is a vector space, and the Bochner integral is a linear operator from this space into $X$.

\subsection{Lipschitz functions and the free space}
Let $(M,d)$ be a metric space and $(X,\|\cdot\|_X)$ be a Banach space. A function $\varphi\colon M\to X$ is said to be Lipschitz if there exists a constant $L>0$ such that
$$\|\varphi(x)-\varphi(y)\|_X \leq L d(x,y)$$
for all $x,y \in M$. The infimum of all such constants $L$ is called the Lipschitz constant of $\varphi$, which is given by
$$\|\varphi\|_{Lip} := \sup \left\{ \frac{\|\varphi(x)-\varphi(y)\|_X}{d(x,y)} : x,y \in M, x \neq y \right\}.$$

The set of all Lipschitz functions from $M$ to $X$, denoted by $\operatorname{Lip}(M,X)$, is a vector subspace of the space of continuous functions $\mathcal{C}(M,X)$. While $\|\cdot\|_{Lip}$ defines a seminorm on $\operatorname{Lip}(M,X)$, it fails to be a norm as it vanishes on constant functions. In order to obtain a normed space, we consider $M$ as a pointed metric space $(M, d, \theta)$ by fixing a distinguished point $\theta \in M$. Now, define
$$\operatorname{Lip}_0(M,X) = \{ \varphi \in \operatorname{Lip}(M, X) : \varphi(\theta) = 0 \}.$$
In this setting, $\|\cdot\|_{Lip}$ is indeed a norm, and $(\operatorname{Lip}_0(M,X), \|\cdot\|_{Lip})$ is a Banach space. In the scalar case, we denote $\operatorname{Lip}_0(M) := \operatorname{Lip}_0(M, \mathbb{R})$. This space is often referred to as the Lipschitz dual of $M$ and is denoted by $M^\#$. For a comprehensive treatment of Lipschitz functions, the reader may refer to \cite{cobzaslipschitz}.

For each $x \in M$, let $\delta_x \colon \operatorname{Lip}_0(M) \to \mathbb{R}$ be the evaluation functional defined by $\delta_x(f) := f(x)$. These functionals are linear and continuous, thus $\delta_x \in \operatorname{Lip}_0(M)^*$. By equipping the dual space $\operatorname{Lip}_0(M)^*$ with its standard dual norm, the Lipschitz-free space over $M$ (or simply the free space of $M$) is defined as 
$$\mathcal{F}(M) := \overline{\operatorname{span}} \{ \delta_x : x \in M \}.$$
The induced norm on $\mathcal{F}(M)$, denoted by $\|\cdot\|_{\mathcal{F}}$, is given by
$$\|m\|_{\mathcal{F}} = \sup \left\{ \langle m, \varphi \rangle : \varphi \in \operatorname{Lip}_0(M), \|\varphi\|_{Lip} \leq 1 \right\}$$
for $m \in \mathcal{F}(M)$. Under this norm, $(\mathcal{F}(M), \|\cdot\|_{\mathcal{F}})$ is a Banach space. Furthermore, the property $d(x,y) = \|\delta_x - \delta_y\|_{\mathcal{F}}$ holds for all $x,y \in M$. Consequently, the map $\delta \colon M \to \mathcal{F}(M)$ defined by $x \mapsto \delta_x$ is an isometric embedding of the metric space $M$ into the Banach space $\mathcal{F}(M)$.

We present in what follows some relevant results concerning the free space.
\begin{theorem}
For any pointed metric space $M$, there exists an isometric isomorphism $\mathcal{F}(M)^* \cong \operatorname{Lip}_0(M)$.
\end{theorem}
\begin{theorem}[Universal property] \label{thm:universal}
Let $M$ be a pointed metric space and $X$ be a Banach space. For every Lipschitz map $\varphi \in \operatorname{Lip}_0(M, X)$, there exists a unique bounded linear operator $T \colon \mathcal{F}(M) \to X$ such that $\varphi = T \circ \delta$. Moreover, this operator satisfies $\|T\| = \|\varphi\|_{Lip}$.
\end{theorem}
\begin{theorem}\label{thm:universaltwo}
Let $(M, d_M, \theta_M)$ and $(N, d_N, \theta_N)$ be pointed metric spaces, and let $g \colon M \to N$ be a Lipschitz map such that $g(\theta_M) = \theta_N$. Let $\delta_M$ and $\delta_N$ denote the canonical embeddings of $M$ and $N$ into $\mathcal{F}(M)$ and $\mathcal{F}(N)$, respectively. Then, there exists a unique bounded linear operator $T_g \colon \mathcal{F}(M) \to \mathcal{F}(N)$ such that $T_g \circ \delta_M = \delta_N \circ g$. Furthermore, the operator satisfies $\|T_g\| = \|g\|_{Lip}$.
\end{theorem}

A comprehensive study of free spaces can be found in \cite{weaverlipschitz}. We conclude this section by presenting alternative descriptions of the free space. For this part, we will follow \cite{godefroylipschitz}.

Let $(Z, \|\cdot\|_Z)$ be a Banach space. Since the identity map $\mathrm{Id}_Z$ is $1$-Lipschitz and satisfies $\mathrm{Id}_Z(0)=0$, the Universal Property (Theorem \ref{thm:universal}) guarantees the existence of a unique bounded linear operator $\beta \colon \mathcal{F}(Z) \to Z$ with $\|\beta\|=1$ such that $\beta(\delta_z) = z$ for all $z \in Z$. This operator $\beta$ is known as the \emph{barycenter map}.

Define the operator $P \colon \mathcal{F}(Z) \to \mathcal{F}(Z)$ as $P := \delta \circ \beta$,
\begin{center}
\begin{tikzcd}[column sep = 3em, row sep = 2.6em]
    Z \arrow[r, "Id"] \arrow[d, hook, "\delta"] & Z \arrow[d, hook, "\delta"] \\
    \mathcal F(Z) \arrow[r, "P"] \arrow[ru, "\beta"] & \mathcal F(Z)
\end{tikzcd}
\end{center}
\noindent and note that 
$$P^2 = (\delta \circ \beta) \circ (\delta \circ \beta) = \delta \circ (\beta \circ \delta) \circ \beta = \delta \circ \mathrm{Id}_Z \circ \beta = P.$$
Consequently, for any $m \in \mathcal{F}(Z)$, we have the decomposition $m=Pm+(m-Pm)$ and hence $\mathcal F(Z)=\operatorname{Im}P+\ker\beta$ since
$$\beta(m-Pm)=\beta m-\beta(Pm)=\beta m-(\beta\delta)\beta m=\beta m-\beta m=0.$$

Suppose now that $m \in \mathcal{F}(Z)$ admits two such decompositions: $$m=Pu+h=P\tilde u+\tilde h,$$
where $u,\tilde u\in\mathcal F(Z)$ and $h,\tilde h\in\ker\beta$. Then, using the definition of $P$, we compute
$$\begin{aligned}
Pu &= P^2u = \delta(\beta(Pu)) = \delta(\beta(Pu + h)) \\
   &= \delta(\beta(P\tilde u + \tilde h)) = \delta(\beta(P\tilde u)) = P\tilde u,
\end{aligned}$$
It follows that $h = \tilde h$, and hence the decomposition is unique. Therefore, we conclude
$$\mathcal F(Z)=\operatorname{Im}P\oplus \ker\beta.$$
However, note that for every $z\in Z$ it holds $\delta_z=\delta(\beta(\delta_z))=P\delta_z$.That is, $\delta(Z)\subseteq\operatorname{Im}P$.
Moreover, for all $m\in\operatorname{Im}P$ there exist $u\in\mathcal F(Z)$ such that $m=Pu=\delta(\beta(u))$. Since $\beta(u)\in Z$ it follows that $m\in\delta(Z)$, and so $\operatorname{Im}P\subseteq \delta(Z)$.
We conclude that $\operatorname{Im}P=\delta(Z)$ and therefore
$$\mathcal F(Z)=\delta(Z)\oplus\ker\beta.$$

Although each element $m \in \mathcal F(Z)$ admits a unique decomposition $m = \delta_z + h$ with $z \in Z$ and $h \in \ker\beta$, this does not imply that $\delta(Z)$ is a linear subspace of $\mathcal F(Z)$, since $P$ is not linear. Nevertheless, the quotient $\mathcal F(Z) / \ker \beta$ can be naturally identified with $\delta(Z)$ through the bijection
\begin{align*}
    \Lambda: \mathcal F(Z) / \ker \beta & \to \delta(Z) \\
    m + \ker \beta & \mapsto P(m).
\end{align*}
This map is an isometry when the quotient is endowed with its usual norm $\| m + \ker \beta \| = \inf \{ \| m + h \|_{\mathcal F(Z)}: h \in \ker \beta \}$.
In this way, $\Lambda$ induces a linear structure of $\mathcal F(Z) / \ker \beta$ onto $\delta(Z)$, which is given by $\delta_x \boxplus \delta_y = \delta_{x+y}$ and $\lambda \boxdot \delta(x) = \delta_{\lambda x}$.
This linear structure coincides with the one of $Z$ via the barycenter map $\beta$, but differs from the one inherited from $\mathcal F(Z)$ to $\delta(Z) \subseteq \mathcal F(Z)$.
This discrepancy between both structures, and the failure of $\delta$ and $P$ being linear, is captured by $\ker \beta$, as shown in the next proposition.

\begin{proposition}
    Let $(Z,\|\cdot\|_Z)$ be a Banach space, and consider the barycenter map $\beta\colon\mathcal F(Z)\to Z$. Then, $\ker\beta=\overline{\mathrm{span}}\{\delta_x+\delta_y-\delta_{x+y}\colon x,y\in Z\}$.
\end{proposition}
\begin{proof}
    Let $E := \overline{\mathrm{span}}\{\delta_x+\delta_y-\delta_{x+y} \colon x,y \in Z\}$. We first note that $E \subseteq \ker\beta$. Indeed, by the definition of the barycenter map $\beta$, we have$$\beta(\delta_x+\delta_y-\delta_{x+y}) = x + y - (x+y) = 0,$$and the inclusion follows by the linearity and continuity of $\beta$.

    To prove the reverse inclusion, consider the canonical quotient map $\pi\colon \mathcal F(Z) \to \mathcal F(Z)/E$. Since $E$ is a subspace of $\ker\beta$, the Universal Property of Quotient spaces yields a (unique) bounded linear continuous map $\tilde\beta\colon\mathcal F(Z)/E\to Z$ such that $\beta=\tilde\beta\circ \pi$.

    Define now $L := \pi\circ \delta\colon Z \to \mathcal F(Z)/E$. Observe that $L$ is additive, since for any $x, y \in Z$ we get
    $$L(x+y) = \pi(\delta_{x+y}) = \pi(\delta_x) + \pi(\delta_y) = L(x) + L(y),$$
    which holds since $\delta_x+\delta_y-\delta_{x+y} \in E$. Furthermore, $L$ is Lipschitz because it is the composition of the 1-Lipschitz map $\pi$ and the isometry $\delta$. Hence, $L$ is linear and continuous.

    For any $z \in Z$ we have
    $$(L \circ \tilde{\beta})(\pi(\delta_z)) = L(\beta(\delta_z)) = L(z) = \pi(\delta_z).$$
    By linearity and density, $L \circ \tilde{\beta}$ is the identity operator on $\mathcal F(Z)/E$. In particular, this implies that the operator $\tilde{\beta}$ is injective.

    Finally, take any $m \in \ker\beta$. Then $\tilde{\beta}(\pi(m)) = \beta(m) = 0$. By the injectivity of $\tilde{\beta}$ we deduce $\pi(m) = 0$. That is, $m \in E$. Therefore, $\ker\beta \subseteq E$, concluding the proof.
\end{proof}

\begin{proposition}
    Let $(Z,\|\cdot\|_Z)$ be a Banach space, and consider the barycenter map $\beta\colon\mathcal F(Z)\to Z$. Then, $\ker\beta=(Z^*)_\perp$.
\end{proposition}
\begin{proof}
    Recall that for any bounded and linear map between Banach spaces $T\colon X\to Y$, it holds $\ker T=(\operatorname{Im}T^*)_\perp$. Then, if we prove that $\operatorname{Im}\beta^*=Z^*$ we get the result.

    Note that $\beta^*$ is the adjoint map $\beta^*\colon Z^*\to\mathcal F(Z)^*$ defined by
    $$\langle m,\beta^*z^*\rangle=\langle\beta(m),z^*\rangle \quad \text{for all} \ z^*\in Z^*, m\in\mathcal F(Z).$$
    In particular, for any $z\in Z$ yields
    $$\langle \delta_z, \beta z^*\rangle=\langle \beta(\delta_z),z^*\rangle=\langle z,z^*\rangle.$$

    On the other hand, viewing $z^*$ as an element of $\mathrm{Lip}_0(Z)$ (since every bounded linear functional on $Z$
    is Lipschitz and vanishes at $0$), we also have $z^*(\delta_z)=z^*(z)$. Therefore, $\beta^*z^*$ and $z^*$
    coincide on $\{\delta_z:z\in Z\}$. Since both are continuous linear functionals on $\mathcal{F}(Z)$ and $\operatorname{span}\{\delta_z:z\in Z\}$ is dense in $\mathcal{F}(Z)$, we conclude $\beta^*z^*=z^*$ in $\mathrm{Lip}_0(Z)$.

    This shows that $\beta^*$ is the inclusion $Z^*\hookrightarrow\mathrm{Lip}_0(Z)$. In particular,
    $\operatorname{Im}\beta^*=Z^*$, and hence
    $\ker\beta=(\operatorname{Im}\beta^*)_\perp=(Z^*)_\perp$.
\end{proof}



  
\section{The free integral}\label{sec:freeint}
Let $(\Omega, \Sigma,\mu)$ be a finite measure space, and let $(M,d,\theta)$ be a pointed metric space. In this section, we aim to define an integral for metric-valued functions $f:\Omega\to M$. To this end, we will first recall some concepts of measurability, before providing an explicit and motivated construction of our proposed integral. Although our definition can be characterised by standard means, we will present an exhaustive motivation for the sake of completeness. Having established integration for such metric-valued functions, we will study the extension of some well-known properties of vector-valued integrals and demonstrate some relations between the two concepts.


Let us start this section by recalling some basic concepts about integrability. The formal construction of any integral typically begins with the class of simple functions. A vector-valued simple function is defined as a finite linear combination of indicator functions of measurable sets. However, they can also be defined in a similar way in metric spaces, without the need for linear combinations.
\begin{definition}
    A function $S\colon\Omega\to M$ is called simple if there exists $x_1,\ldots,x_n\in M$ and a pairwise disjoint measurable partition $A_1,\ldots,A_n\in\Sigma$ 
    of $\Omega$ such that $$S(\omega)=x_i \quad \text{if} \ \omega\in A_i, \ \text{for} \ i=1,\ldots,n.$$
    We denote the set of all simple functions from $\Omega$ to $M$ by $\mathrm{SF}(M)$.
\end{definition}

Furthermore, another preliminary requirement for the formal development of the integral is the rigorous definition of measurability. The well-known notions of strong and Borel measurability are associated with concepts such as convergence and open sets, and can therefore be defined in the usual way when the function takes values in a topological space. However, weak measurability is associated with the dual of a Banach space, so a different definition is required in this instance.
\begin{definition}
A function $f\colon\Omega\to M$ is said to be:
\begin{itemize}
    \item Borel measurable if $f^{-1}(U)\in\Sigma$ for every open set $U\subseteq M$. The space of such functions is denoted by $L^0_\mathcal B(M)$.
    \item Strongly measurable (or Bochner measurable) if there exists a sequence of simple functions convergent pointwise to $f$ $\mu$-a.e..  We write $L^0_S(M)$ for the space of strongly measurable functions.
    \item Weakly measurable if for every $\varphi\in\mathrm{Lip}_0(M)$ the composition $\varphi\circ f\colon\Omega\to\mathbb R$ is Lebesgue measurable. The space of weakly measurable functions is denoted by $L^0_W(M)$.
\end{itemize}
\end{definition}

Note that the simple functions match all three concepts of measurability. Indeed, for any $S\in \text{SF}(M)$ and any open set $U\subset M$, the preimage $S^{-1}(U)$ is the union of some of the measurable sets that defines $S$. Hence, $S^{-1}(U)$ is in $\Sigma$ and therefore every simple function is Borel measurable. Moreover, simple functions are strongly measurable by definition (consider the constant sequence of the simple function itself). Finally, for each $\varphi\in\text{Lip}_0(M)$, the composition $\varphi\circ S$ is a real-valued simple function, hence Lebesgue measurable. Thus, simple functions are also weakly measurable.

\begin{proposition}\label{prop:measureinclusions}
    Let $(M,d,\theta)$ be a pointed metric space.
    \begin{enumerate}
        \item Borel and weak measurability coincide:
        $$L^0_{\mathcal B}(M)=L^0_{W}(M).$$
        \item If the measure space $(\Omega,\Sigma,\mu)$ is complete, then every
        strongly measurable function is Borel (weakly) measurable:
        $$L^0_{S}(M)\subseteq L^0_{\mathcal B}(M).$$
    \end{enumerate}
    Therefore, when $(\Omega,\Sigma,\mu)$ is complete we obtain the chain
    $L^0_{S}(M)\subseteq L^0_{\mathcal B}(M)=L^0_{W}(M)$.
\end{proposition}

\begin{proof}
    \emph{(1)} Suppose first that $f\in L^0_{\mathcal B}(M)$, and fix $\varphi\in\operatorname{Lip}_0(M)$ and an open set $U\subseteq\mathbb R$. Since $\varphi$ is continuous, $\varphi^{-1}(U)$ is open in $M$. Therefore,
    $$(\varphi\circ f)^{-1}(U)=f^{-1}\bigl(\varphi^{-1}(U)\bigr)\in\Sigma.$$
    Thus $\varphi\circ f$ is measurable for every $\varphi\in\operatorname{Lip}_0(M)$,
    so $f\in L^0_W(M)$.

    Conversely, let $f\in L^0_W(M)$ and let $U\subseteq M$ be open. We assume without loss of generality $U\neq M$, so that $U^c\neq\emptyset$. The function
    $\varphi(\cdot)=d(\cdot,U^c)-d(\theta,U^c)$ belongs to $\operatorname{Lip}_0(M)$
    and satisfies
    $$U=\{x\in M:d(x,U^c)>0\}=\varphi^{-1}\bigl(-d(\theta,U^c),+\infty\bigr).$$
    Hence
    $$f^{-1}(U)=(\varphi\circ f)^{-1}\bigl(-d(\theta,U^c),+\infty\bigr)\in\Sigma,$$
    since the interval is open in $\mathbb R$ and $\varphi\circ f$ is measurable by
    the weak measurability of $f$. Therefore $f\in L^0_{\mathcal B}(M)$, and (1) follows.

    \emph{(2)} Let $f\in L^0_S(M)$, and let $\{S_k\}_k\subseteq\mathrm{SF}(M)$ converge
    pointwise to $f$ on $\Omega\setminus N$, where $\mu(N)=0$. Fix
    $\varphi\in\operatorname{Lip}_0(M)$. Note that each $\varphi\circ S_k$ is a real-valued simple
    function, and by continuity of $\varphi$ we get $\varphi\circ S_k\to\varphi\circ f$ on
    $\Omega\setminus N$. That is, $\varphi\circ f$ is pointwise limit $\mu$-a.e. of measurable functions. Since $(\Omega,\Sigma,\mu)$ is complete, $\varphi\circ f$
    is  measurable. As $\varphi\in\operatorname{Lip}_0(M)$ was arbitrary we conclude $f\in L^0_W(M)$.
\end{proof}

The equality $L^0_B(M)=L^0_W(M)$ can be understood as a Lipschitz analogue of the following classical fact: in metric spaces, the Baire $\sigma$-algebra (the smallest $\sigma$-algebra making all continuous real-valued functions measurable) coincides with the Borel $\sigma$-algebra (see \cite[Section 6.3]{bogachevmeasure}). In our setting, the role of continuous functions is played by Lipschitz functions in $\mathrm{Lip}_0(M)$. The equality $L^0_B(M)=L^0_W(M)$ holds because this subclass already generates the Borel $\sigma$-algebra on $M$. We present this result for the sake of completeness.
\begin{proposition}
    Let $(M,d,\theta)$ be a pointed metric space, and let $\mathcal B(M),\mathcal W(M)$ denote the Borel and Lipschitz $\sigma$-algebras, that is, $\mathcal W(M)$ is the smallest $\sigma$-algebra making all functions in $\mathrm{Lip}_0(M)$ measurable. Then, $\mathcal B(M)=\mathcal W(M)$.
\end{proposition}
\begin{proof}
    It is obvious that $\mathcal W(M)\subseteq\mathcal B(M)$ since Lipschitz functions are continuous.

    For the reverse inclusion, let $F\subseteq M$ be a closed subset and define $g(\cdot):=d(\cdot, F)-d(\theta, F)$. Note that $g(\theta)=0$, and the triangular inequality yields
    $$|g(x)-g(y)|=|d(x,F)-d(y,F)|\le d(x,y).$$
    That is, $g\in\mathrm{Lip}_0(M)$. Since $x\in F$ if and only if $d(x,F)=0$, equivalently $g(x)=-d(\theta, F)$, we obtain $F=g^{-1}(\{-d(\theta,F)\})$. As $g$ is $\mathcal{W}(M)$-measurable by definition and $\{-d(\theta,F)\}$ is a closed subset of $\mathbb{R}$, it follows that $F\in\mathcal{W}(M)$. Since $\mathcal{B}(M)$ is generated by the closed subsets of $M$, we conclude
    $\mathcal{B}(M)\subseteq\mathcal{W}(M)$.
\end{proof}

On the other hand, the inclusion $L_W^0(M)\subseteq L_S^0(M)$ may not hold. As in the standard case when $M$ is a vector space, there are several examples of weakly measurable functions that are not strongly measurable. The reader can refer to Example 5 on page 43 of \cite{diestelmeasures}. Nevertheless, these two spaces may be equal depending on $M$. The next result can be found in \cite[Proposition~1.9]{vakhaniya1987probability}:
\begin{proposition}\label{prop:separable}
    Let $(M,d,\theta)$ be a separable pointed metric space. Then for every Borel measurable
    function $f\colon\Omega\to M$ there exists a sequence $\{S_k\}_k\subseteq\mathrm{SF}(M)$
    such that $S_k(\omega)\to f(\omega)$ for every $\omega\in\Omega$.
\end{proposition}

\begin{remark}\label{rmk:separable}
    Combining Propositions~\ref{prop:measureinclusions} and \ref{prop:separable}, if $M$ is
    separable and $(\Omega,\Sigma,\mu)$ is complete, then the three measurabilities coincide:
    $$L^0_{S}(M)=L^0_{\mathcal B}(M)=L^0_{W}(M).$$
    Moreover, the separability hypothesis is automatic in a frequent setting: if $\Omega$ is a
    Polish space and $\Sigma$ its Borel $\sigma$-algebra, then every Borel measurable
    $f\colon\Omega\to M$ has separable range \cite[Proposition~1.11]{vakhaniya1987probability}.
    The conclusion above then applies to $f$ regarded as a map into $\overline{f(\Omega)}$,
    even when $M$ itself is not separable.
\end{remark}

\subsection{Definition and properties} \label{sec:integral}
\begin{definition}\label{def:sfintegral}
    Let $S$ be a simple function for $x_1,\ldots,x_n\in M$ and $A_1,\ldots, A_n\in\Sigma$ pairwise disjoint. We define its integral over $E\in\Sigma$ as follows:
    $$F_M-\int_E Sd\mu=\sum_{i=1}^n \mu(A_i\cap E)\delta_{x_i}\in\mathcal F(M).$$
    We will write $F-\int_ESd\mu$ if the metric space is understood, or simply $\int_ESd\mu$ when there is no confusion regarding the free integral.
\end{definition}


\begin{proposition}\label{prop:sfintegral-welldefined}
    The free integral of a simple function in Definition~\ref{def:sfintegral} does not depend
    on the representation of $S$.
\end{proposition}

\begin{proof}
    Let $S$ be a simple function who admits two representations
    $$\begin{cases}
        S(\omega)=x_i & \omega\in A_i, i=1,\ldots,n\\
        S(\omega)=y_i & \omega\in B_j, j=1,\ldots,m
    \end{cases}$$
    where $x_i,y_j\in M$ and $\{A_i\}_{i=1}^n$, $\{B_j\}_{j=1}^m$ are measurable partitions of
    $\Omega$. Define $C_{i,j}:=A_i\cap B_j$ and note that the additivity of $\mu$ yields
    $$\mu(A_i\cap E)=\sum_{j=1}^m\mu(C_{i,j}\cap E),\qquad
      \mu(B_j\cap E)=\sum_{i=1}^n\mu(C_{i,j}\cap E).$$
    Furthermore, for each pair $(i,j)$,
    $$\mu(C_{i,j}\cap E)\,\delta_{x_i}=\mu(C_{i,j}\cap E)\,\delta_{y_j}.$$
    Indeed, if $A_i\cap B_j=\emptyset$ both sides vanish. If $A_i\cap B_j\neq\emptyset$, then
    for any $\omega\in C_{i,j}$ we have $x_i=S(\omega)=y_j$. Hence $\delta_{x_i}=\delta_{y_j}$ and
    combining these identities we coclude
    $$\sum_{i=1}^n\mu(A_i\cap E)\,\delta_{x_i}
        =\sum_{i=1}^n\sum_{j=1}^m\mu(C_{i,j}\cap E)\,\delta_{x_i}
        =\sum_{j=1}^m\sum_{i=1}^n\mu(C_{i,j}\cap E)\,\delta_{y_j}
        =\sum_{j=1}^m\mu(B_j\cap E)\,\delta_{y_j}.$$
\end{proof}

The following result states a common technique when working with simple functions. We omit the proof due to its straightforwardness.
\begin{lemma}\label{lemma:compat}
    Let $S_1$ be a simple function for $x_1,\ldots,x_n\in M$ and $A_1,\ldots,A_n\in\Sigma$. Let $S_2$ be another simple function for $y_1,\ldots,y_m\in M$ and $B_1,\ldots,B_m\in\Sigma$. Then,
    $$\alpha\int_E S_1d\mu +\beta \int_E S_2d\mu=\sum_{i=1}^n\sum_{j=1}^m \mu(C_{i,j}\cap E)(\alpha\delta_{x_i}+\beta\delta_{y_j}),$$
    for all $E\in\Sigma$ where $C_{i,j}:=A_i\cap B_j$.
\end{lemma}
\begin{corollary}\label{cor:normintdist}
    For simple functions $S_1, S_2$ and measurable set $E\in\Sigma$ it holds
    $$\left\|\int_ES_1d\mu-\int_ES_2d\mu\right\|_\mathcal F\leq\int_Ed(S_1,S_2)d\mu.$$
\end{corollary}

We first introduce the definition of the integral for strongly measurable functions $f\colon\Omega\to M$. Such a function $f$ can be expressed as the pointwise limit of a sequence of simple functions $\{S_k\}_k$. This naturally leads us to the following question: If $\{S_k\}_k$ converges pointwise $\mu$-almost everywhere, does the sequence of its integrals converge in $\mathcal F(M)$? If this convergence holds, we can define the integral of $f$ as the limit of the integrals of $S_k$, thus linking pointwise convergence in $\Omega$ with uniform convergence in $\mathcal F(M)$. The following result covers this question.

\begin{proposition}\label{prop:convergencefree}
    Let $\{S_k\}_k$ be a sequence of simple functions pointwise convergent $\mu$-a.e., and $E\in\Sigma$. If $\sup_k d(\theta,S_k(\omega))$ is integrable over $E$, then the sequence $\{\int_ES_kd\mu\}_k$ is convergent in $\mathcal F(M)$.
\end{proposition}
\begin{proof}
    Fix a measurable set $E\in\Sigma$, and consider for each $N\in\mathbb N$ the function $g_N\colon\Omega\to\mathbb R$ defined by $g_N(\omega)=\sup_{n,m\geq N} d(S_n(\omega),S_m(\omega))$. 
    Note that $g_N$ is measurable and
    $$0\leq g_N(\omega)\leq 2 \sup_k d(S_k(\omega),\theta)=:G(\omega).$$
    That is, $g_N$ is integrable and dominated by the integrable function $G$. Furthermore, there exist a null set $E_0\in\Sigma$ such that $\{S_k(\omega)\}_k$ is a Cauchy sequence for all $\omega\in\Omega\setminus E_0$. Therefore, $g_N$ is pointwise convergent to 0 for each $\omega\in\Omega\setminus E_0$. By the Dominated Convergence Theorem it holds
    $$\lim\limits_N \int_Eg_Nd\mu=0.$$
    Consequently, for all $\varepsilon>0$ there exist $N_0\in\mathbb N$ such that
    \begin{equation}\label{eq:intdistb}
        \int_Ed(S_n,S_m)d\mu\leq \int_Eg_{N_0}d\mu<\varepsilon
    \end{equation}
    if $n,m\geq N_0$. On the other hand, by Corollary \ref{cor:normintdist} we get
    \begin{equation}\label{eq:normdifb}
        \left\|\int_E S_nd\mu-\int_E S_md\mu\right\|_\mathcal F\leq \int_Ed(S_n,S_m)d\mu.
    \end{equation}
    From \eqref{eq:intdistb} and \eqref{eq:normdifb} we conclude that the sequence $\{\int_ES_kd\mu\}_k$ is Cauchy in the Banach space $\mathcal F(M)$, and therefore convergent.
\end{proof}

\begin{proposition}\label{prop:welldefined}
    Let $f$ be a strongly measurable function, and let $\{S_k\}_k$
    and $\{S'_k\}_k$ be two sequences of simple functions pointwise
    convergent to $f$ $\mu$-a.e.\ such that $\sup_k d(\theta,S_k)$
    and $\sup_k d(\theta,S'_k)$ are both integrable over $E\in\Sigma$.
    Then,
    $$
        \lim_k\int_E S_k\,d\mu = \lim_k\int_E S'_k\,d\mu.
    $$
\end{proposition}
\begin{proof}
    By Proposition \ref{prop:convergencefree}, both limits exist in $\mathcal{F}(M)$. For each $k\in\mathbb N$, by Corollary \ref{cor:normintdist}
    we have
    \[
        \left\|\int_E S_k\,d\mu - \int_E S'_k\,d\mu\right\|_F
        \leq \int_E d(S_k, S'_k)\,d\mu.
    \]
    Since both $S_k$ and $S'_k$ converge pointwise to $f$
    $\mu$-a.e., we have $d(S_k,S'_k)\to 0$ pointwise $\mu$-a.e.
    Moreover,
    \[
        d(S_k(\omega),S'_k(\omega))
        \leq d(S_k(\omega),\theta) + d(\theta,S'_k(\omega))
        \leq \sup_k d(\theta,S_k(\omega))
        + \sup_k d(\theta,S'_k(\omega)),
    \]
    which is integrable. By the Dominated Convergence Theorem,
    $\lim_k\int_E d(S_k,S'_k)\,d\mu=0$, and the result follows.
\end{proof}

The previous results ensure that the following definition is consistent.
\begin{definition}
    Let $f$ be a strongly measurable function. We say that $f$ is free integrable over $E\in\Sigma$ if there exist a sequence of simple functions $\{S_k\}_k$  pointwise convergent to $f$ $\mu$-a.e. such that
    $\sup_kd(\theta,S_k)$ is integrable over $E$. In that case, the integral is defined as $$F_M-\int_Efd\mu = \lim\limits_k \int_ES_kd\mu.$$
    We write the space of free integrable functions as $L_\mathcal F^1(\mu,M)$.
\end{definition}

The following result characterises free integrability of a strongly measurable function in terms of the function itself, without requiring the explicit construction of an approximating sequence of simple functions.
\begin{theorem}\label{thm:freeintchar}
    Let $f$ be a strongly measurable function. Then, $f$ is free integrable if and only if $d(\theta,f)$ is integrable.
    \begin{proof}
        Since $f$ is strongly measurable, there exist a sequence of simple functions $\{S_k\}_k$ pointwise convergent to $f$ $\mu$-a.e.. Then, $d(\theta,f)$ is measurable since it is the pointwise limit of the real simple functions $d(\theta, S_k)$.
        
        Suppose first that $f$ is free integrable and suppose without loss of generality that $\sup_kd(\theta, S_k)$ is integrable. Then, observe that
        $$0\leq d(\theta,f)=\lim_kd(\theta,S_k)\leq \sup_kd(\theta, S_k).$$
        From this we conclude that $d(\theta,f)$ is integrable since $\sup_k d(\theta, S_k)$ is integrable.

        On the other hand, suppose that $d(\theta,f)$ is integrable. Let $E_0\in\Sigma$ be the null set such that $S_k(\omega)$ converges to $f(\omega)$ for all $\omega\in\Omega\setminus E_0$. For each such $\omega$ and $\varepsilon>0$ there exist $k(\omega,\varepsilon)\in\mathbb N$ so that $d\bigl(f(\omega),S_k(\omega)\bigr)<\varepsilon$ for all $k\geq k(\omega,\varepsilon)$.
        Define now the following sequence of simple functions:
        $$\tilde{S}_k(\omega):=\begin{cases}
            S_k(\omega) & \text{if} \ \ d\bigl(\theta, S_k(\omega)\bigr)\leq d\bigl(\theta,f(\omega)\bigr)+1,\\
            \theta & \text{if} \ \ d\bigl(\theta, S_k(\omega)\bigr)> d\bigl(\theta,f(\omega)\bigr)+1.
        \end{cases}$$
        For $k\geq\max\{k(\omega,\varepsilon),k(\omega,1)\}$ note that $d\bigl(\tilde{S}_k(\omega),f(\omega)\bigr)=d\bigl(S_k(\omega),f(\omega)\bigr)<\varepsilon,$
        and therefore $\{\tilde{S}_k\}_k$ is pointwise convergent to $f$ $\mu$-a.e..
        To finish the proof, let us to show that $\sup_kd\bigl(\tilde{S}_k,\theta\bigr)$ is integrable. Observe that
        $$d\bigl(\tilde{S}_k(\omega),\theta\bigr)=\begin{cases}
            d\bigl(S_k(\omega),\theta\bigr) & \text{if} \ \ d\bigl(\theta, S_k(\omega)\bigr)\leq d\bigl(\theta,f(\omega)\bigr)+1,\\
            0 & \text{if} \ \ d\bigl(\theta, S_k(\omega)\bigr)> d\bigl(\theta,f(\omega)\bigr)+1,
        \end{cases}$$
        and therefore $d\bigl(\tilde{S}_k(\omega),\theta\bigr)\leq d\bigl(f(\omega),\theta\bigr)+1$. From this and the integrability of $d\bigl(f(\omega),\theta\bigr)$ we conclude that $\sup_k d\bigl(\tilde{S}_k(\omega),\theta\bigr)$ is integrable.
    \end{proof}
\end{theorem}

The free integral of a metric-valued function $f$ can be linked with the Bochner integral of a specific function related with $f$ in $\mathcal F(M)$. To see that, recall that a strongly measurable vector-valued function is Bochner integrable if and only if its norm is integrable \cite[Chapter 2, Theorem 2]{diestelmeasures}
\begin{proposition}\label{prop:freebochner}
    For any strongly measurable function (in the metric-valued sense) $f\colon\Omega\to M$ it holds that $\delta_f : \Omega \to \mathcal F(M)$ is strongly measurable (in the vector-valued sense).
    In that case, $f$ is free integrable if and only if $\delta_f : \Omega \to \mathcal F(M)$ is Bochner integrable.
    Furthermore, for any $E \in \Sigma$,
    $$F-\int_Efd\mu=B-\int_E\delta_fd\mu.$$
\end{proposition}
\begin{proof}
    As $f$ is strongly measurable, consider $\{S_k\}_k\subseteq\mathrm{SF}(M)$ pointwise convergent to $f$ $\mu$-a.e.. Note that $\delta_{S_k}$ is a simple function in $\mathcal F(M)$ and $$\lim_k\|\delta_{S_k}-\delta_f\|_\mathcal{F}=\lim_kd\bigl(S_k,f\bigr)=0, \quad \mu-\text{a.e.}.$$
    Therefore, $\delta_f$ is strongly measurable (in the vector-valued sense).
    
    On the other hand, as $f$ is free integrable, by Theorem \ref{thm:freeintchar} we know that $d(\theta,f)$ is integrable. Since
    $\|\delta_f\|_\mathcal{F}=d(f,\theta)$, we conclude that $\delta_f$ is Bochner integrable.
    For the reciprocal, is enought to consider the same formula, but apply the charactrizaion in the inverse order.
\end{proof}
    

By identifying the free integral with the Bochner integral, we can transfer and generalize classical results from integration theory to the setting of metric-valued functions. For instance, it is worth mentioning the theorems of convergence (monotone and dominated). However, as we are focusing on the metric space itself, the subsequent analysis will use the free space as a framework for studying the integral in a metric context.
\begin{theorem}\label{thm:int_dual}
    Let $(X,\|\cdot\|_X)$ be a Banach space. For any $\varphi\in\mathrm{Lip}_0(M,X)$ consider the bounded linear map $T_\varphi\colon\mathcal F(M)\to X$ given by Theorem \ref{thm:universal} such that $\varphi=T_\varphi\circ\delta$. Then, for all free integrable function $f$ and $E\in\Sigma$ it holds
    $$T_\varphi\left( F-\int_E fd\mu\right)=B -\int_E\varphi\circ fd\mu.$$
\end{theorem}
\begin{proof}
    First, note that a direct calculation gives the result for any simple function $S$ and measurable set $E\in\Sigma$:
    \begin{align*}
        B - \int_E\varphi\circ Sd\mu&=\sum_{j=1}^n\mu(E_i\cap E)\varphi(x_i)\\&=\sum_{j=1}^n\mu(E_i\cap E)T_\varphi(\delta_{x_i})\\&=T_\varphi\left(\sum_{j=1}^n\mu(E_i\cap E)\delta_{x_i}\right)=T_\varphi\left( F-\int_E Sd\mu\right).
    \end{align*}

    Consider now a free integrable function $f$ and a sequence of simple functions $S_k$ pointwise convergent to $f$ $\mu$-a.e.. Then, $\varphi\circ S_k$ is pointwise convergent to $\varphi\circ f$ $\mu$-a.e.. Furthermore, recall that $\sup_kd(S_k,\theta)$ is Bochner integrable and $$\|\varphi\circ S_k(\omega)\|_X\leq \|\varphi\|_{Lip}d(S_k(\omega),\theta)\leq \|\varphi\|_{Lip}\sup_kd(S_k(\omega),\theta).$$ That is, $\varphi\circ S_k$ is dominated by an integrable function. Therefore, $\varphi\circ f$ is Bochner integrable and hence
    \begin{align*}
        B - \int_E\varphi\circ fd\mu&=\lim_k B -\int_E\varphi\circ S_kd\mu\\
        &=\lim_k T_\varphi\left(F-\int_ES_kd\mu\right)\\
        &=T_\varphi\left(\lim_k F-\int_ES_kd\mu\right)=T_\varphi\left(F-\int_Efd\mu\right).
    \end{align*}
\end{proof}

\begin{corollary}
    Let $f\colon\Omega\to M$ be a free integrable function, and let $(X,\|\cdot\|_X)$ be a Banach space. For all $\varphi\in\mathrm{Lip}_0(M,X)$ and $E\in\Sigma$ it holds
    $$\left\|B \ \text{-} \int_E \varphi\circ fd\mu\right\|_X\le \|\varphi\|_{Lip}\left\|F \ \text{-} \int_E  fd\mu\right\|_\mathcal{F}.$$
\end{corollary}

\begin{corollary}\label{cor:duality_free_integral}
    Let $f:\Omega\to M$ be a free integrable function. Then, for every
    $\varphi\in\mathrm{Lip}_0(M)$ and every $E\in\Sigma$,
    \begin{equation}\label{eq:free_integral_duality}
        \left\langle F-\int_E f\,d\mu,\,\varphi\right\rangle
        = \int_E \varphi\circ f\,d\mu.
    \end{equation}
\end{corollary}
\begin{proof}
    For $\varphi\in\mathrm{Lip}_0(M)=\mathrm{Lip}_0(M,\mathbb{R})$, the bounded linear operator
    $T_\varphi:\mathcal{F}(M)\to\mathbb{R}$ given by Theorem \ref{thm:universal} satisfies
    $T_\varphi(m)=\langle m,\varphi\rangle$ for all
    $m\in\mathcal{F}(M)$. Therefore, Theorem \ref{thm:int_dual} yields
    \[
        \left\langle F-\int_E f\,d\mu,\,\varphi\right\rangle
        = T_\varphi\left(F-\int_E f\,d\mu\right)
        = B-\int_E \varphi\circ f\,d\mu
        = \int_E \varphi\circ f\,d\mu,
    \]
    where the last equality holds since $\varphi\circ f$ is a
    real-valued integrable function and the Bochner integral coincides
    with the Lebesgue integral in $\mathbb{R}$.
\end{proof}

The previous results describe the action of the free integral through Lipschitz maps into Banach spaces. We now extend this framework to
Lipschitz maps between pointed metric spaces, showing that the free integral is compatible with the linearisation of such maps through
the associated free spaces.

\begin{proposition}\label{prop:universaltwofree}
    Let $(M,d_M, \theta_M)$ and $(N,d_N,\theta_N)$ be pointed metric spaces. Consider a Lipschitz map $g\colon M\to N$ such that $g(\theta_M)=\theta_N$. Then, there exist a bounded linear map $T_g\colon \mathcal F(M)\to\mathcal F(N)$ such that for every free integrable function $f\in L_\mathcal F^1(\mu,M)$ over $E\in\Sigma$ it holds
    $$T_g\left(F_M-\int_E fd\mu\right)=F_N-\int_E g\circ fd\mu.$$
\end{proposition}
\begin{proof}
    The existence of such bounded linear map $T_g$ is given by Theorem \ref{thm:universaltwo}. For such $T_g$ it holds $T_g\circ\delta_M=\delta_N\circ g$. Then, by Proposition \ref{prop:freebochner} we get
    \begin{align*}
        F_N-\int_Eg\circ fd\mu&=B-\int_E\delta_N\circ g\circ fd\mu\\
        &=B-\int_ET_g\circ\delta_M\circ fd\mu\\
        &=T_g\left(B-\int_E\delta_M\circ fd\mu\right)=T_g\left(F_M-\int_Efd\mu\right).
    \end{align*}
\end{proof}

\begin{proposition}
    Let $(M,d,\theta)$ be a pointed metric space. Consider $M_0\subset M$ such that $\theta\in M_0$, and let $i\colon M_0\to M$ be the inclusion map. Then, 
    $$F_{M_0}-\int_Efd\mu=F_M-\int_Ei\circ fd\mu, \quad \text{for all} \ E\in\Sigma.$$
\end{proposition}
\begin{proof}
    Note that $i$ is a Lipschitz map with $\operatorname{Lip}(i) \le 1$ and satisfies $i(\theta)=\theta$. By Proposition \ref{prop:universaltwofree}, there exists a bounded linear map $T_i\colon \mathcal F(M_0)\to\mathcal F(M)$ such that
    $$T_i\left(F_{M_0}-\int_E fd\mu\right)=F_M-\int_E i\circ fd\mu.$$
    This map $T_i$ satisfies $T_i\circ\delta=\delta\circ i$. In particular, for any $x\in M_0$ observe that $T_i(\delta_x)=\delta_{i(x)}$. By linearity and continuity of $T_i$ and density of molecules, we obtain that $T_i$ is the inclusion map between $\mathcal F(M_0)$ and $\mathcal F(M)$. Hence,
    $$F_{M_0}-\int_E fd\mu=F_M-\int_E i\circ fd\mu.$$
\end{proof}

Observe that the construction of the Lipschitz-free space relies on the selection of a distinguished fixed point. However, the concept of free integrability remains independent of this choice. Indeed, if $\theta,\theta'\in M$ are two distinguished points, the triangular inequality yields
$d(f,\theta')\le d(f,\theta)+d(\theta,\theta')$. Hence, $L^1_\mathcal F(\mu, (M,d,\theta))=L^1_\mathcal F(\mu,(M,d,\theta'))$. In the following we present the relation between both free integrals.

\begin{proposition}
     Let $(M,d)$ be a metric space, and let $\theta,\theta'\in M$. There exists a bounded linear operator $T\colon \mathcal F_\theta(M)\to\mathcal F_{\theta'}(M)$ with $\|T\|=1$ such that for any $f\in L_{\mathcal F_\theta}^1(\mu, M)$ free integrable function over $E\in \Sigma$ it holds
     $$F_{\theta'}-\int_E fd\mu=T\left(F_\theta-\int_Efd\mu\right)+\mu(E)\delta_\theta^{\theta'},$$
     where $\delta_\theta^{\theta'}$ is the evaluation functional of $\theta$ is the free space of $(M,d,\theta')$.
\end{proposition}
\begin{proof}
    Consider $\varphi\colon (M,d,\theta)\to\mathcal F_{\theta'}(M)$ defined as $\varphi(x):=\delta_x^{\theta'}-\delta_{\theta}^{\theta'}$. We observe that $\varphi\in\operatorname{Lip}_0(M,\mathcal F_{\theta'}(M))$ since $$\varphi(\theta)=\delta_\theta^{\theta'}-\delta_\theta^{\theta'}=0_{\mathcal F_{\theta'}(M)}=\delta_{\theta'}^{\theta'},$$
    and
    $$\|\varphi(x)-\varphi(y)\|_{\mathcal F_{\theta'}(M)}=\|\delta_x^{\theta'}-\delta_{y}^{\theta'}\|_{\mathcal F_{\theta'}(M)}=d(x,y),$$ for all $x,y\in M$. Therefore, the Universal Property gives a linear continous map $T\colon \mathcal F_\theta(M)\to\mathcal F_{\theta'}(M)$ such that $\|T\|\le 1$ and $T(\delta_x^\theta)=\varphi(x)=\delta_x^{\theta'}-\delta_{\theta}^{\theta'}$. Moreover, by Theorem \ref{thm:int_dual} we obtain
    $$T\left(F_\theta-\int_Efd\mu\right)=B-\int_E\varphi\circ fd\mu=B-\int_E\delta_{f}^{\theta'}-\delta_{\theta}^{\theta'}d\mu=F_{\theta'}-\int_Efd\mu-\mu(E)\delta_{\theta}^{\theta'}.$$
\end{proof}

To conclude this section, we present a characterization of the norm of the free integral expressed in terms of the integrand.
\begin{lemma}\label{lemma:simnorm}
    Let $S$ be a simple function for $x_1,\ldots,x_n\in M$ and $A_1,\ldots, A_n\in\Sigma$. Then,
    $$\left\|\int_E Sd\mu\right\|_\mathcal F=\sum_{i=1}^n\mu(A_i\cap E)d(\theta,x_i).$$
    \begin{proof}
        For each $\varphi\in\text{Lip}_0(M)$ and $x\in M$ we get $(\delta_x-\delta_\theta)(\varphi)=\varphi(x)-\varphi(\theta)=\varphi(x)=\delta_x$. Therefore, by the isometry between $M$ and $\mathcal F(M)$ we get $\|\delta_x\|_\mathcal F=\|\delta_x-\delta_\theta\|_\mathcal F=d(\theta,x)$. Consequently,
        \begin{equation}\label{eq:ubint}
            \left\|\int_E Sd\mu\right\|_\mathcal F=\left\|\sum_{i=1}^n\mu(A_i\cap E)\delta_{x_i}\right\|_\mathcal F\leq \sum_{i=1}^n\mu(A_i\cap E)\|\delta_{x_i}\|_\mathcal F =\sum_{i=1}^n\mu(A_i\cap E)d(\theta,x_i).
        \end{equation}

        On the other hand, consider $\tilde\varphi(\cdot):=d(\theta,\cdot)$. Note that $\tilde\varphi(\theta)=0$ and $\|\tilde\varphi\|_{Lip}=1$. Therefore
        \begin{align*}\label{eq:lbint}
            \left\|\int_E Sd\mu\right\|_\mathcal F & = \sup \left\{\left|\sum_{i=1}^n \mu(A_i\cap E)\varphi(x_i)\right|:\varphi\in\text{Lip}_0(M), \, \|\varphi\|_{Lip} \leq 1\right\} \\ & \geq 
            \left|\sum_{i=1}^n\mu(A_i\cap E)\tilde\varphi(x_i)\right|
            \\ & = \sum_{i=1}^n\mu(A_i)d(\theta,x_i).
        \end{align*}
        From this and \eqref{eq:ubint} we conclude the result.
    \end{proof}
\end{lemma}

\begin{proposition}\label{prop:freenormdist}
    Let $f$ be a free integrable function. Then, for all $E\in\Sigma$ holds
    $$\left\|\int_E f d\mu\right\|_\mathcal F=\int_Ed\bigl(\theta,f\bigr)d\mu.$$
    \begin{proof}
        Note that the result holds for all simple functions by Lemma \ref{lemma:simnorm}. Indeed, if $S$ is a simple function for $x_1,\ldots,x_n\in M$ and $A_1\ldots,A_n\in\Sigma$, then $$d\bigl(\theta,S(\omega)\bigr)=\sum_{i=1}^nd\bigr(\theta,x_i\bigr)\chi_{A_i}(\omega),$$
        which is a simple function from $\Omega$ to $\mathbb R$. Consequently, by Lemma \ref{lemma:simnorm} we conclude
        $$\int_Ed\bigl(\theta, S\bigr)d\mu=\sum_{i=1}^n\mu(A_i\cap E)d\bigl(\theta, x_i\bigr)=\left\|\int_E S d\mu\right\|_\mathcal F.$$

        Consider now a free integrable function $f$. Then, by Corollary \ref{cor:duality_free_integral} we get
        \begin{align*}
            \left\|\int_Efd\mu\right\|_\mathcal F&=\sup_{\|\varphi\|_{Lip}\le 1}\left| \left<\int_Efd\mu,\varphi\right>\right|\\
            &=\sup_{\|\varphi\|_{Lip}\le 1}\left| \int_E\varphi\circ fd\mu\right|\\
            & \leq \sup_{\|\varphi\|_{Lip}\le 1} \|\varphi\|_{Lip} \int_E d(\theta,f)d\mu=\int_E d(\theta,f)d\mu.
        \end{align*}
        On the other hand, consider $\tilde\varphi(\cdot)=d(\theta,\cdot)$. Then, $\tilde\varphi\in\operatorname{Lip}_0(M)$ and $\|\tilde\varphi\|_{Lip}=1$. Therefore,
        $$ \left\|\int_Efd\mu\right\|_\mathcal F=\sup_{\|\varphi\|_{Lip}\le 1} \int_E\varphi(f)d\mu\geq \int_Ed(\theta,f)d\mu,$$
        from which the result can be concluded.
    \end{proof}
\end{proposition}

\subsection{The free integral in the Banach case}
When the metric space is a Banach space $(X,\|\cdot\|_X)$ with base point $\theta=0_X$, the free integral can be compared directly with the classical Bochner integral. The aim of this section is therefore to study, in the context of Banach spaces, how the free integral and the Bochner integral compare and what differences exist between them. In this setting, the barycenter map $\beta:\mathcal{F}(X)\to X$ and the decomposition $\mathcal{F}(X)=\delta(X)\oplus\ker\beta$ provide a natural framework for analysing the additional information carried by the free integral.

\begin{proposition}\label{prop:betabochner}
    Let $(X,\|\cdot\|_X)$ be a Banach space, and let $f\colon \Omega\to X$ be a Bochner integrable function over $E\in\Sigma$. Consider the metric space $(M,d)$ induced by $X$, and let $\theta=0_X$. Then, $f$ is free integrable and
    $$\beta\left(F\ \text{-}\int_Efd\mu\right)=B \ \text{-}\int_Efd\mu.$$
\end{proposition}
\begin{proof}
    Note that $f$ is strongly measurable and $\|f\|_X$ is integrable in $E$ since $f$ is Bochner integrable. Furthermore, for all $\omega\in\Omega$ it holds
    $$d\bigl(\theta,f(\omega)\bigr)=\|0_X-f(\omega)\|_X=\|f(\omega)\|_X.$$
    Consequently, $f$ is free integrable.

    On the other hand, consider the identity map $\mathrm{Id}$ $\in X$. Since $\mathrm{Id}\in\mathrm{Lip}_0(M)$, by Theorem \ref{thm:int_dual}, there exist a bounded linear operator $\beta\colon\mathcal F(X)\to X$ such that $\beta(\delta_x)=\mathrm{Id}(x)=x$ for all $x\in X$. Moreover,
    $$\beta\left(F\ \text{-}\int_Efd\mu\right)=B\ \text{-}\int_E\mathrm{Id}\circ fd\mu= B\ \text{-}\int_Efd\mu.$$
\end{proof}
Recall that $\mathcal F(X)=\delta(X)\oplus \ker\beta$, where $\beta$ is the barycenter. Then, for any $f\in L^1_\mathcal F(\mu,M)$, its free integral can be expressed as
\begin{equation}\label{eq:freesum}
    F-\int_\Omega fd\mu=\delta_{x_0}+m_f,
\end{equation}
for some $x_0\in X$ and $m_f\in\ker\beta.$ Evaluating $\beta$ in \eqref{eq:freesum} we get
$$\beta\left(F-\int_\Omega fd\mu\right)=\beta(\delta_{x_0})+\beta(m_f)=x_0.$$
Therefore, by Proposition \ref{prop:betabochner} we conclude
$$x_0=B-\int_\Omega fd\mu.$$

\begin{proposition}\label{prop:normmf}
    Suppose that $\mu(\Omega)=1$. Naming $x_0=B-\int_\Omega fd\mu$, for every $f\in L^1_\mathcal F(\mu,M)$ it holds
    $$\|m_f\|_\mathcal F=\int_\Omega\|f(\omega)-x_0\|_Xd\mu(\omega).$$
\end{proposition}
\begin{proof}
    Since $m_f=F-\int_\Omega fd\mu-\delta_{x_0}$, then
    $$\|m_f\|_\mathcal F=\sup_{\|\varphi\|_{Lip}\le 1}\langle m_f,\varphi\rangle=\sup_{\|\varphi\|_{Lip}\le 1}\int_\Omega \varphi\circ fd\mu-\varphi(x_0).$$

    Since $\mu(\Omega)=1$, for any $\varphi\in\operatorname{Lip}_0(M)$ such that $\|\varphi\|_{Lip} \leq 1$ yields
    $$\int_\Omega \varphi(f(\omega))d\mu(\omega)-\varphi(x_0)=\int_\Omega \varphi(f(\omega))-\varphi(x_0)d\mu(\omega)\le \int_\Omega \|f(\omega)-x_0\|_X.$$
    Hence, $\|m_f\|_\mathcal F\le \int_\Omega \|f(\omega)-x_0\|_X$. To see the reverse inequality, consider $\tilde\varphi(x)=\|x-x_0\|_X$. It is obvious that $\tilde\varphi\in\operatorname{Lip}_0(X)$ and $\|\tilde\varphi\|_{Lip}=1$. Therefore, we conclude
    $$\|m_f\|_\mathcal F\ge \langle m_f,\tilde\varphi\rangle =\int_\Omega \|f(\omega)-x_0\|_Xd\mu(\omega).$$
\end{proof}

The results above show that the free integral of a Banach-valued function $f$ decomposes into two components with distinct roles.
The term $\delta_{x_0}$, where $x_0=B-\int_\Omega f\,d\mu$, represents the average or centre of mass of $f$ through the Bochner integral. The term $m_f\in\ker\beta$ quantifies the dispersion of
$f$ around that centroid. By Proposition \ref{prop:normmf}, its norm
equals the mean absolute deviation $\int_\Omega\|f-x_0\|\,d\mu$.
Consequently, the free integral provides a strictly more comprehensive mathematical object than the Bochner integral. Whereas the latter reduces the distributional information of $f$ to a single point in the target space, the free integral preserves the intrinsic spread of the distribution.

Furthermore, recall the Jensen inequality. If $(\Omega,\Sigma,\mu)$ is a probability space, then for any $g\in L^1(\mu)$ and $\varphi$ a real convex function, then
$$\varphi\left(\int_\Omega gd\mu\right)\le\int_\Omega \varphi\circ gd\mu.$$
Then, for any convex Lipschitz function $\varphi \in \operatorname{Lip}_0(X)$, the duality pairing
$$\langle m_f, \varphi \rangle = \int_{\Omega} \varphi(f(\omega)) d\mu(\omega) - \varphi(x_0)$$
corresponds to the Jensen inequality gap. Therefore, the element $m_f \in \mathcal{F}(X)$ is an universal representative of these gaps.

\begin{example}\label{ex:dispersion}
    Consider $X=\mathbb{R}^2$ with the Euclidean norm and
    $\theta=0$. Let $\Omega=\{1,2,3,4\}$ equipped with the uniform
    probability measure $\mu(\{i\})=1/4$.

    Define $f:\Omega\to\mathbb{R}^2$ and $g:\Omega\to\mathbb{R}^2$ by
    \[
        f(1)=(1,0),\quad f(2)=(-1,0),\quad f(3)=(0,1),\quad f(4)=(0,-1),
    \]
    \[
        g(1)=\tfrac{1}{\sqrt{2}}(1,1),\quad
        g(2)=\tfrac{1}{\sqrt{2}}(-1,1),\quad
        g(3)=\tfrac{1}{\sqrt{2}}(-1,-1),\quad
        g(4)=\tfrac{1}{\sqrt{2}}(1,-1).
    \]
    A direct computation shows that they share the
    same Bochner integral,
    \[
        B-\int_\Omega f\,d\mu
        = B\text{-}\int_\Omega g\,d\mu = (0,0) = x_0,
    \]
    and the same mean absolute deviation around $x_0$,
    \[
        \int_\Omega\|f(\omega)-x_0\|\,d\mu
        = \int_\Omega\|g(\omega)-x_0\|\,d\mu = 1.
    \]
    In particular, $\|m_f\|_{\mathcal{F}} = \|m_g\|_{\mathcal{F}} = 1$
    by Proposition \ref{prop:normmf}. However, $m_f\neq m_g$.
    Indeed, the function $\varphi:\mathbb{R}^2\to\mathbb{R}$ defined
    by $\varphi(x_1,x_2)=|x_1|$ belongs to $\mathrm{Lip}_0
    (\mathbb{R}^2)$ with $\|\varphi\|_{\mathrm{Lip}}=1$, and
    \[
        \langle m_f,\varphi\rangle
        = \int_\Omega\varphi\circ f\,d\mu - \varphi(x_0)
        = \tfrac{1}{4}(1+1+0+0) = \tfrac{1}{2},
    \]
    while
    \[
        \langle m_g,\varphi\rangle
        = \int_\Omega\varphi\circ g\,d\mu - \varphi(x_0)
        = \tfrac{1}{4}\!\left(\tfrac{1}{\sqrt{2}}+\tfrac{1}{\sqrt{2}}
        +\tfrac{1}{\sqrt{2}}+\tfrac{1}{\sqrt{2}}\right)
        = \tfrac{1}{\sqrt{2}}.
    \]
    This illustrates that the component $m_f\in\ker\beta$ encodes the
    \emph{directional distribution} of the values of $f$ around its
    centroid, not merely the scalar magnitude of the spread. The
    functions $f$ and $g$ disperse their mass at the same distance
    from the origin, but along different directions, and the free
    integral detects this difference while the Bochner integral and
    the mean deviation do not.
\end{example}

The previous example shows that the component $m_f$ can distinguish between functions that the Bochner integral and the mean deviation
cannot separate. A natural question is then: when does $m_f$ carry no additional information, that is, when does the free integral reduce to the Bochner integral? The following result provides a complete answer.

\begin{proposition}\label{prop:mf_zero}
    Let $(\Omega,\Sigma,\mu)$ be a probability space, $(X,\|\cdot\|)$
    a Banach space, and $f\in L^1_{\mathcal{F}}(\mu,X)$. Then the
    following are equivalent:
    \begin{enumerate}[(i)]
        \item $f = x_0$ $\mu$-a.e. for some $x_0\in\mathcal F(M)$.
        \item $m_f = 0$, i.e.,
              $F-\int_\Omega f\,d\mu = \delta_{x_0}$.
        \item The Jensen gap vanishes for every convex Lipschitz
              function: $\int_\Omega\varphi\circ f\,d\mu = \varphi(x_0)$
              for all convex $\varphi\in\mathrm{Lip}_0(X)$.
    \end{enumerate}
\end{proposition}
\begin{proof}
    The implication (i)$\Rightarrow$(ii) is immediate from Proposition \ref{prop:normmf}. Moreover, from the identity
    $\langle m_f,\varphi\rangle = \int_\Omega\varphi\circ f\,d\mu -
    \varphi(x_0)$ it follows
    (ii)$\Rightarrow$(iii). For (iii)$\Rightarrow$(ii), take
    $\varphi(\cdot)=\|\cdot - x_0\|_X$, which is convex and $1$-Lipschitz. Then,
    \[
        0 = \varphi(x_0)=\int_\Omega \|f(\omega)-x_0\|_X\,d\mu(\omega),
    \]
    and since the integrand is non-negative, we conclude $f=x_0$ $\mu$-a.e.
\end{proof}

Having characterised the dispersion component $m_f$ and identified when it vanishes, we now study how it transforms under composition
with Lipschitz maps between Banach spaces. Given a Lipschitz map $g\colon X\to Y$, the linearisation $T_g\colon\mathcal F(X)\to\mathcal F(Y)$ relates the free integrals of $f$ and $g\circ f$ by Proposition \ref{prop:universaltwofree}. The next result describes how this relationship decomposes the dispersion components $m_f$ and $m_{g\circ f}$.

\begin{proposition}\label{prop:dispersioncomp}
    Let $(\Omega,\Sigma,\mu)$ be a probability space, and let
    $(X,\|\cdot\|_X)$, $(Y,\|\cdot\|_Y)$ be Banach spaces with base
    points $\theta_X=0_X$ and $\theta_Y=0_Y$. Let
    $f\in L^1_{\mathcal{F}}(\mu,X)$ and let $g:X\to Y$ be a Lipschitz
    map with $g(0_X)=0_Y$ such that $g\circ f\in
    L^1_{\mathcal{F}}(\mu,Y)$. Write
    \[
        F\text{-}\int_\Omega f\,d\mu = \delta_{x_0}+m_f,\qquad
        F\text{-}\int_\Omega g\circ f\,d\mu = \delta_{y_0}+m_{g\circ f},
    \]
    where $x_0=B\text{-}\int_\Omega f\,d\mu$,\;
    $y_0=B\text{-}\int_\Omega g\circ f\,d\mu$, and
    $m_f\in\ker\beta_X$,\; $m_{g\circ f}\in\ker\beta_Y$. Then,
    \begin{equation}\label{eq:mgf_decomposition}
        m_{g\circ f} = T_g(m_f) + \delta_{g(x_0)} - \delta_{y_0},
    \end{equation}
    where $T_g:\mathcal{F}(X)\to\mathcal{F}(Y)$ is the linearisation
    of $g$ given by Theorem \ref{thm:universaltwo}.
\end{proposition}
\begin{proof}
    By Proposition \ref{prop:universaltwofree}, we have
    \[
        T_g\!\left(F-\int_\Omega f\,d\mu\right)
        = F-\int_\Omega g\circ f\,d\mu.
    \]
    Expanding both sides using the decomposition
    $\mathcal{F}(Z)=\delta(Z)\oplus\ker\beta_Z$ and the linearity of
    $T_g$, we obtain
    \[
        T_g(\delta_{x_0}) + T_g(m_f) = \delta_{y_0} + m_{g\circ f}.
    \]
    Since $T_g\circ\delta_X=\delta_Y\circ g$ by Theorem \ref{thm:universaltwo}, we have
    $T_g(\delta_{x_0})=\delta_{g(x_0)}$. From this the result can be derived.
\end{proof}

\begin{corollary}\label{cor:mgfchar}
    Under the hypotheses of Proposition \ref{prop:dispersioncomp}, the following
    statements are equivalent:
    \begin{enumerate}[(i)]
        \item $m_{g\circ f} = T_g(m_f)$.
        \item $\delta_{g(x_0)} = \delta_{y_0}$.
        \item $g(x_0) = y_0$, that is,
        \[
            g\!\left(B\text{-}\int_\Omega f\,d\mu\right)
            = B\text{-}\int_\Omega g\circ f\,d\mu.
        \]
    \end{enumerate}
    In particular, this holds whenever $g$ is linear.
\end{corollary}
\begin{proof}
    The equivalence (i)$\Leftrightarrow$(ii) is immediate from
    \eqref{eq:mgf_decomposition}. The equivalence
    (ii)$\Leftrightarrow$(iii) follows from the fact that $\delta$ is
    an isometric embedding and hence injective.

    If $g$ is linear, then $g\circ f$ is Bochner integrable and
    \[
        B\text{-}\int_\Omega g\circ f\,d\mu
        = g\!\left(B\text{-}\int_\Omega f\,d\mu\right) = g(x_0),
    \]
    so (iii) holds.
\end{proof}

\begin{remark}
    Proposition \ref{prop:dispersioncomp} yields an estimate
    for the norm of $m_{g\circ f}$. By the triangle inequality and
    \eqref{eq:mgf_decomposition},
    \[
        \|m_{g\circ f}\|_{\mathcal{F}(Y)}
        \leq \|T_g(m_f)\|_{\mathcal{F}(Y)}
        + \|\delta_{g(x_0)}-\delta_{y_0}\|_{\mathcal{F}(Y)}
        = \|T_g(m_f)\|_{\mathcal{F}(Y)} + \|g(x_0)-y_0\|_Y.
    \]
    Since $\|T_g\|=\|g\|_{\mathrm{Lip}}$, we obtain
    \begin{equation}\label{eq:norm_mgf_bound}
        \|m_{g\circ f}\|_{\mathcal{F}(Y)}
        \leq \|g\|_{\mathrm{Lip}}\,\|m_f\|_{\mathcal{F}(X)}
        + \|g(x_0)-y_0\|_Y.
    \end{equation}
    If $g$ is linear, by Corollary \ref{cor:mgfchar} and
    \eqref{eq:norm_mgf_bound} we deduce
    \[
        \|m_{g\circ f}\|_{\mathcal{F}(Y)}
        \leq \|g\|_{\mathrm{Lip}}\,\|m_f\|_{\mathcal{F}(X)}.
    \]
    That is, linear maps cannot amplify dispersion beyond their Lipschitz constant. For nonlinear $g$, the additional term accounts for the distortion introduced by the nonlinearity of $g$ at the centroid $x_0$.
\end{remark}

\subsection{The metric space of integrable functions}
In the classical theory of Bochner integration over a Banach space $X$, the space $L^1(\mu,X)$ of integrable functions inherits also a Banach space structure. In our metric framework, the absence of an underlying linear structure implies that the space $L^1_{\mathcal{F}}(\mu, M)$ cannot be equipped with a natural vector space structure. Nevertheless, we show that it
admits a natural metric, analogous to the
$L^1$-norm, and that the resulting metric space inherits some fundamental properties from $M$. Furthermore, we will study also how the resulting metric space of free integrable functions can be linearised through its free space. From now on we will understand that in $L_\mathcal F^1(\mu,M)$ there is the equivalence relation $f\sim g$ if and only if $f=g$ $\mu$-a.e.

\begin{proposition}
    The function $\rho\colon L_\mathcal F^1(\mu,M)\times L_\mathcal F^1(\mu,M)\to \mathbb R$ defined as $\rho(f,g)=\int_\Omega d(f,g)d\mu$ is a metric. Moreover, if the metric space $(M,d)$ is complete, then $(L_\mathcal F^1(\mu,M),\rho)$ is a complete metric space too.
    \begin{proof}
        Since the metric properties of $\rho$ are straightforward, we will only prove that $(L_\mathcal F^1(\mu,M),\rho)$ is complete whenever $(M,d)$ is complete. 
        
        Consider a Cauchy sequence of free integrable functions $\{f_n\}_n$. Therefore, for any $\varepsilon>0$ there exist $N(\varepsilon)\in\mathbb N$ such that $\rho(f_n,f_m)<\varepsilon$ if $n,m\geq N(\varepsilon)$. In particular, we can consider for each $k\in\mathbb N$ some $N_k\in\mathbb N$ such that $\rho(f_n,f_m)<2^{-k}$ if $n,m\geq N_k$ and $N_k\leq N_{k+1}$. Define now 
        $$g(\omega):=\sum_{k=1}^\infty d\bigl(f_{N_{k+1}}(\omega),f_{N_k}(\omega)\bigr).$$
        By the Monotone Convergence Theorem we obtain
        \begin{align*}
            \int_\Omega gd\mu&=\int_\Omega \sum_{k=1}^\infty d\bigl(f_{N_{k+1}}(\omega),f_{N_k}(\omega)\bigr)d\mu\\
            &=\sum_{k=1}^\infty\int_\Omega d\bigl(f_{N_{k+1}}(\omega),f_{N_k}(\omega)\bigr)d\mu\\
            &=\sum_{k=1}^\infty \rho(f_{N_{k+1}},f_{N_k})= \leq 1.
        \end{align*}
        Consequently, there exist a null set $\Omega_0$ such that $g(\omega)$ is finite for all $\omega\in\Omega\setminus\Omega_0$. That is, the series defined by $g$ is convergent and hence $\{f_{N_k}(\omega)\}_k$ is a Cauchy sequence in $M$. Since $M$ is complete, there exist its limit. Define $f(\omega)=\lim_kf_{N_k}(\omega)$ for all $\omega\in\Omega\setminus\Omega_0.$

        It is obvious that $f$ is strongly measurable since it is the pointwise limit of strongly measurable functions. On the other hand, recall that $d\bigl(\theta,f_{N_1}\bigr)$ is integrable since $f_{N_1}$ is free integrable. Moreover, for each $p\in\mathbb N$ note that by the triangular we can get
        $$d(f_{N_1},f_{N_p})\leq\sum_{j=1}^{p-1} d(f_{N_j},f_{N_{j+1}}).$$
        The continuity of $d$ yields
        \begin{equation}\label{eq:gbound}
            d(f_{N_1}(\omega),f(\omega))=\lim_pd(f_{N_1}(\omega),f_{N_p}(\omega))\leq \lim_p\sum_{j=1}^{p-1}d(f_{N_j}(\omega),f_{N_{j+1}}(\omega))=g(\omega),
        \end{equation}
        for all $\omega\in\Omega\setminus\Omega_0.$ Consequently,
        $$d\big(\theta,f\bigr)\leq d\bigl(\theta,f_{N_1}\bigr)+d\bigl(f_{N_1},f\bigr)\leq d\bigl(\theta,f_{N_1}\bigr)+g .$$
        That is, $d\big(\theta,f\bigr)$ is integrable and hence $f$ is free integrable.

        Finally, to complete the proof let us show that $f$ is the limit of the sequence $\{f_n\}_n$. By a similar argument presented in \eqref{eq:gbound} we can bound
        $$d(f_{N_k}(\omega),f(\omega))\leq g(\omega),$$
        for each $k\in\mathbb N$ and $\omega\in\Omega\setminus \Omega_0$. Since the sequence of measurable functions $d(f_{N_k}(\omega),f(\omega))$ is pointwise convergent to 0 $\mu$-a.e. and bounded by the integrable function $g$, by the Dominated Convergence Theorem we get $\lim_k \rho(f_{N_k},f)=0$. From this and 
        $$\rho(f_n,f)\leq \rho(f_n,f_{N_k})+\rho(f_{N_k},f),$$
        we conclude that $\{f_n\}_n$ converges to the free integrable function $f$.
    \end{proof}
\end{proposition}

\begin{remark}
   Consider a compact metric space $M$. In this case, recall that $M$ is also complete and separable. Thus, strong and weak measurability are equivalent. Furthermore, there exists a constant $K>0$ such that $d(a,b)\leq K$ for all $a,b\in M$. In particular, for any measurable function $f\colon M\to\Omega$ we have that $\int_\Omega d(f(\omega),\theta)d\mu\leq K\mu(\Omega)<\infty$.  That is, every measurable function is integrable and the set of all measurable functions is complete with respect to the metric $\rho$.
\end{remark}

The lack of vector space structure in $L^1_\mathcal F(\mu,M)$ means that the free integral operator $f\in L^1_\mathcal F(\mu,M)\mapsto F-\int_\Omega fd\mu$ cannot be linear, as the standard case of Lebesgue or Bochner integrals in Banach spaces. However, the free integral operator can be linearized as we will show in the next. From the definition of $L^1_\mathcal F(\mu, M)$ as a metric space follows some results that will be useful for this purpose.
\begin{corollary}\label{cor:freeintembedding}
    $L_\mathcal F^1(\mu,M)$ embeds isometrically into $L^1(\mu,\mathcal F(M))$.
\end{corollary}
\begin{proof}
    For all free integrable function $f\in L^1_\mathcal F(\mu, M)$, by Proposition \ref{prop:freebochner} we get that $\delta_f\in\ L^1(\mu,\mathcal F(M))$. Therefore, the map $J\colon L_\mathcal F^1(\mu,M)\to L^1(\mu,\mathcal F(M))$ defined as $(Jf)(\omega):=\delta_{f(\omega)}$ for all $\omega\in\Omega$ is well defined. Furthermore, by standard arguments can be proved that $\|Jf-Jg\|_{L^1}=\rho(f,g)$, and hence $L_\mathcal F^1(\mu,M)$ embeds isometricaly into $L^1(\mu,\mathcal F(M))$.
\end{proof}

\begin{corollary}\label{cor:filipschitz}
    The free integral operator is 1-Lipschitz.
\end{corollary}
\begin{proof}
    For all $f,g\in L^1_\mathcal F(\mu, M)$, by Proposition \ref{prop:freebochner} and standard techniques we get $$\left\|\int_\Omega fd\mu-\int_\Omega gd\mu\right\|_\mathcal F\le \rho(f,g).$$
    That is, the free integral operator is Lipschitz with Lipschitz constant no greater than 1. Finally, by Proposition \ref{prop:freenormdist} we deduce that the Lipschitz constant is equal to 1.
\end{proof}

In the following we provide the canonical linearisation of the integration process through the free space of $L^1_{\mathcal{F}}(\mu,M)$. To see this, let $\eta:L^1_{\mathcal{F}}(\mu,M)\to\mathcal{F}(L^1_{\mathcal{F}}(\mu,M))$ denote its canonical isometric embedding into its free space.
\begin{theorem}\label{thm:linearisation}
    There exists a unique bounded linear
    operator
    $$
    I:\mathcal{F}(L^1_{\mathcal{F}}(\mu,M))\to\mathcal{F}(M)
    $$
    with $\|I\|=1$ such that
    \begin{equation}\label{eq:linearisation}
        F-\int_\Omega f\,d\mu = I(\eta(f))
        \quad\text{for all }
        f\in L^1_{\mathcal{F}}(\mu,M).
    \end{equation}
    Moreover, $I$ admits the factorisation $I=B\circ\tilde{J}$,
    where $B:L^1(\mu,\mathcal{F}(M))\to\mathcal{F}(M)$ is the
    Bochner integral operator and
    $\tilde{J}:\mathcal{F}(L^1_{\mathcal{F}}(\mu,M))\to
    L^1(\mu,\mathcal{F}(M))$ is the linearisation of the isometric
    embedding $J$ from Corollary \ref{cor:freeintembedding}.
\end{theorem}
\begin{proof}
    By Corollary \ref{cor:filipschitz}, the free integral operator is a Lipschitz map from $L^1_{\mathcal{F}}(\mu,M)$ to $\mathcal{F}(M)$ with Lipschitz constant $1$. Then, the Universal Property yields a unique bounded linear operator $I$ with $\|I\|=1$ satisfying \eqref{eq:linearisation}.

    For the factorisation, the isometric embedding $J:L^1_{\mathcal{F}}(\mu,M)\to L^1(\mu,\mathcal{F}(M))$ from
    Corollary \ref{cor:freeintembedding} satisfies $J(\theta)=\delta_\theta=0$ in $L^1(\mu,\mathcal{F}(M))$, so by
    the Universal Property there exists a unique bounded linear operator
    $\tilde{J}:\mathcal{F}(L^1_{\mathcal{F}}(\mu,M))\to
    L^1(\mu,\mathcal{F}(M))$ with $\|\tilde{J}\|=1$ and
    $\tilde{J}\circ\eta=J$. Now, for every
    $f\in L^1_{\mathcal{F}}(\mu,M)$, note that
    \[
        (B\circ\tilde{J})(\eta(f))
        = B(Jf)
        = B(\delta\circ f)
        = B-\int_\Omega\delta_{f(\omega)}\,d\mu(\omega)
        = F-\int_\Omega f\,d\mu
        = I(\eta(f)).
    \]
    Since both $I$ and $B\circ\tilde{J}$ are bounded linear operators, and $\eta(L^1_{\mathcal{F}}(\mu,M))$ is a dense subspace of
    $\mathcal{F}(L^1_{\mathcal{F}}(\mu,M))$, we conclude $I=B\circ\tilde{J}$.
\end{proof}

The factorisation $I=B\circ\tilde{J}$ can be summarised in the
following commutative diagram:
$$
\begin{tikzcd}[row sep=large, column sep=large]
    L^1_{\mathcal{F}}(\mu,M)
        \arrow[d,"\eta"']
        \arrow[rd,"J"]
    \\
    \mathcal{F}(L^1_{\mathcal{F}}(\mu,M))
        \arrow[r,"\tilde{J}"']
    & L^1(\mu,\mathcal{F}(M))
        \arrow[r,"B"']
    & \mathcal{F}(M)
\end{tikzcd}
$$

\begin{corollary}\label{cor:adjoint}
    The adjoint operator
    $I^*:\mathrm{Lip}_0(M)\to
    \mathrm{Lip}_0(L^1_{\mathcal{F}}(\mu,M))$
    is given by
    \begin{equation}\label{eq:adjoint_formula}
        (I^*\varphi)(f) = \int_\Omega\varphi\circ f\,d\mu,
    \end{equation}
    for all $\varphi\in\mathrm{Lip}_0(M)$ and
    $f\in L^1_{\mathcal{F}}(\mu,M)$.
\end{corollary}

\begin{proof}
    By the identification
    $\mathcal{F}(H)^*\cong\mathrm{Lip}_0(H)$ for any pointed metric
    space $H$, the adjoint $I^*$ maps
    $\mathrm{Lip}_0(M)$ into
    $\mathrm{Lip}_0(L^1_{\mathcal{F}}(\mu,M))$. For every
    $\varphi\in\mathrm{Lip}_0(M)$ and
    $f\in L^1_{\mathcal{F}}(\mu,M)$, using
    \eqref{eq:linearisation} and
    Corollary~\ref{cor:duality_free_integral} we compute
    \[
        (I^*\varphi)(f)
        = \langle\eta(f),I^*\varphi\rangle
        = \langle I(\eta(f)),\varphi\rangle
        = \left\langle F\text{-}\int_\Omega f\,d\mu,\,
        \varphi\right\rangle
        = \int_\Omega\varphi\circ f\,d\mu.
        \qedhere
    \]
\end{proof}

\section{Geometric aspects of the free integral}\label{sec:geometry}
In this section, we connect the analytic framework of metric-valued integration developed in Section 3 with geometric properties of Lipschitz-free spaces. In particular, we will study the extremal structure of its unit ball $B_{\mathcal F(M)}$. Recall that it remains as an open question whether all extreme points of $B_{\mathcal F(M)}$ are elementary molecules $m_{xy} = (\delta_x - \delta_y)/d(x,y)$ (see \cite{aliagaconvex} and the references therein).
Recently, in \cite{aliagaconvex} is introduced the concept of convex integrals of molecules in $\mathcal F(M)$ as a continuous counterpart to convex series, building on the de Leeuw representation of functionals on $\mathrm{Lip}_0(M)$. An element $m\in\mathrm{Lip}_0(M)$ is called a convex integral of molecules if it admits an optimal de Leeuw representation that is concentrated on the off-diagonal set $\tilde M:=\{(x,y)\in M\times M\colon x\neq y\}$. A key geometric result in \cite{aliagaconvex} regarding such representations is that every extreme point of $B_{\mathcal F(M)}$ that is a convex integral of molecules must be an elementary molecule \cite[Theorem 6.1]{aliagaconvex}.
Our aim is to show that the free integral provides a natural source of convex integrals of molecules. Specifically, for any $f\in L_\mathcal F^1(\mu, M)$, the pushforward measure $\nu=f\#\mu$ on $M$ induces a functional $\hat{\nu}\in \mathcal F(M)$ that coincides with the free integral of $f$. Under appropriate completeness assumptions on $M$, this pushforward is a Radon measure, and the resulting functional is a convex integral of molecules by \cite[Theorem 3.1]{aliagaconvex}. This connection allows us to transfer geometric information from the theory of de Leeuw representations to the metric-valued integration setting, and conversely, to use integration-theoretic tools to study the extremal structure of $\mathcal{F}(M)$. To clarify this, we first recall several preliminary definitions and results.

Let $(M,d,\theta)$ be a pointed metric space. The de Leeuw operator is the linear isometric embedding of the Lipschitz functions into the bounded continuous functions $\Phi : \mathrm{Lip}_0(M) \to C_b(\tilde{M})$ given by
$$\Phi\varphi(x,y) = \frac{\varphi(x)-\varphi(y)}{d(x,y)}.$$
Since $\tilde{M}$ is generally not compact, we consider its Stone-Čech compactification $\gamma\tilde{M}$, which allows us to isometrically identify $C_b(\tilde{M})$ with $C(\gamma\tilde{M})$. By the Riesz-Markov-Kakutani theorem, the topological dual of $C(\gamma\tilde{M})$ is isometrically isomorphic to the Banach space of Radon measures $\mathcal{M}(\gamma\tilde{M})$, endowed with the total variation norm.

The adjoint operator $\Phi^* : \mathcal{M}(\gamma\tilde{M}) \to \mathrm{Lip}_0(M)^*$ is a surjective, non-expansive map. Its action on $\varphi \in \mathrm{Lip}_0(M)$ is given by
$$\langle\varphi, \Phi^*\lambda\rangle = \int_{\gamma\tilde{M}} \Phi\varphi \, d\lambda.$$

For each functional $\varphi^* \in \mathrm{Lip}_0(M)^*$, there exists a measure $\lambda \in \mathcal{M}(\gamma\tilde{M})$ such that $\Phi^*\lambda = \varphi^*$. Such a measure $\lambda$ is called a de Leeuw representation of $\varphi^*$. When $\|\varphi^*\| = \|\lambda\|$, the representation is called optimal. The set of all non-negative optimal de Leeuw representations is denoted by $\mathcal{M}_{op}(\gamma\tilde{M})$.

Following \cite{aliagaconvex}, an element $m \in \mathcal{F}(M)$ is called a convex integral of molecules if there exists a representation $\lambda \in \mathcal{M}_{op}(\gamma\tilde{M})$ that is strictly concentrated on $\tilde{M}$ (that is, $\lambda \in \mathcal{M}_{op}(\tilde{M})$). This implies that the functional is induced by a measure supported strictly on pairs of distinct points.

We will now show that the free integral is a convex integral of molecules. For $f \in L_{\mathcal{F}}^1(\mu,M)$, the pushforward measure associated to $f$ and $\mu$ is the Borel measure $\nu = f\#\mu$ defined for any Borel set $E \in \mathcal{B}(M)$ as $\nu(E) = \mu(f^{-1}(E))$. Because $f$ is Borel measurable (by Proposition 3.3), $\nu$ is well-defined. Furthermore, $\nu$ is finite since $\mu$ is finite, and it holds
$$\int_\Omega \varphi \circ f \, d\mu = \int_M \varphi \, d(f_{\#}\mu).$$
In the following we show that $\lambda$ is a Radon measure. To see that, we recall the following result (\cite[Theorem 7.1.7]{bogachevmeasure}):
\begin{theorem}\label{thm:radonchar}
    Let $M$ be a metric space. If $M$ is complete and separable, then every Borel measure is Radon.
\end{theorem}
\begin{theorem}
    Let $(M,d,\theta)$ be a complete metric space. If $f\colon\Omega\to M$ is a strongly measurable function, then the pushforward measure $\nu=f\#\mu\colon\mathcal B(M)\to \mathbb R$ is Radon.
\end{theorem}
\begin{proof}
    First, note that $\delta\circ f\colon\Omega\to \mathcal F(M)$ is a strongly measurable vector function. Then, by \cite[Theorem 2, Chapter II]{diestelmeasures} it holds that $\delta\circ f$ is $\mu$-essentially separably valued in $\mathcal F(M)$. That is, there exist $N\in\Sigma$ such that $\mu(N)=0$ and $\delta\circ f(\Omega\setminus N)$ is separable in $\mathcal F(M)$. Since $\delta$ is an isometry, then $f$ is $\mu$-essentially separably valued in $M$. Indeed, let $\{y_n\}_n\subset \delta\circ f(\Omega\setminus N)$ be dense in $\mathcal F(M)$, and consider $x_n\in M$ such that $y_n=\delta_{x_n}$. Then, for every $x\in f(\Omega\setminus N)$ and $\varepsilon>0$ there exist $n_0\in\mathbb N$ such that
    $\|\delta_x-\delta_{x_{n_0}}\|_\mathcal F=\|\delta_x-y_{n_0}\|_\mathcal F<\varepsilon$. Since $\delta$ is an isometry we deduce $d(x,x_{n_0})<\varepsilon$. From this we conclude that $\{x_n\}_n$ is dense in $f(\Omega\setminus N)$ and hence it is separable.

    Define $M_0:=\overline{f(\Omega\setminus N)}$ and note that $M_0$ is separable and complete. Therefore, $\nu|_{M_0}\colon\mathcal B(M_0)\to\mathbb R$ is Radon by Theorem \ref{thm:radonchar}. Finally, note that for every $E\in\mathcal B(M)$ we have
    $$E=(E\cap M_0)\cup (E\setminus M_0).$$
    Since both measurables sets are disjoint and $\nu(E\setminus M_0)=0$, then
    $\nu(E)=\nu|_{M_0}(E\cap M_0)$. That is, $\nu$ is a Radon measure.
\end{proof}

Let $\nu\in\mathcal M(M)$ be a Radon measure over $M$ with finite first moment. That is, $\int_ M d(x,\theta)d\nu(x)<+\infty$. Then, by \cite[Proposition 4.4]{aliaga2023integral}, the measure $\nu$ induced a functional $\hat\nu\in\mathcal F(M)$ such that $$\langle \hat\nu, \varphi\rangle =\int_M\varphi d\nu, \quad \text{for all} \ \varphi\in\mathrm{Lip}_0(M).$$
If $m\in\mathcal F(M)$ is the induced functional for some Radon measure $\nu$, then $m$ is a convex integral of molecules according to \cite[Theorem 3.1]{aliagaconvex}. Observe that for any $\varphi \in \mathrm{Lip}_0(M)$ yields by Corollary \ref{cor:duality_free_integral}
\begin{equation*}
    \left\langle F-\int_\Omega f d\mu, \varphi \right\rangle = \int_\Omega \varphi \circ f \, d\mu = \int_M \varphi \, d(f\#\mu)
\end{equation*}
Therefore, the following result is obtained.
\begin{corollary}
    Let $f\in L_\mathcal F^1(\mu, M)$ be a free integrable function. Then, its free integral coincides with the functional induced by the Radon measure $\nu=f\#\mu$. Consequently, the free integral is a convex integral of molecules with optimal representation $\lambda\in\mathcal M_{\text{op}}(\tilde M)$.
\end{corollary}

The results above show that the free integral naturally produces convex integrals of molecules, and hence elements of $\mathcal F(M)$ that are geometrically relevant in the sense of \cite{aliagaconvex}. We now study the convex structure of $\mathcal F(M)$ to derive properties of the free integral itself.
\begin{theorem}
    Let $f$ be a free integrable function, and let $E \in \Sigma$ be a measurable set such that $\mu(E) > 0$. Suppose that
    $$\mathcal I:=F-\int_Efd\mu$$
    is an extreme point of $B_{\mathcal{F}(M)}$. If $f(\omega)\neq\theta$ for almost every $\omega\in E$, there exist a measurable set $A\subseteq E$ with $\mu(A)>0$ and $p\in M$ such that
    $$f(\omega)=\begin{cases}
        p & \text{if} \ \omega\in A\\
        \theta & \text{if} \ \omega\in E\setminus A\\
    \end{cases}$$
    for almost every $\omega\in E$. Furthermore, $d(p,\theta)=1/\mu(A)$.
\end{theorem}
\begin{proof}
    Define $A=\{\omega\in E\colon d(f(\omega),\theta)>0\}$. When $\mu(A)=0$ we get the trivial case $f=\theta$ $\mu$-a.e. in $E$. Hence, suppose without loss of generality that $\mu(A)>0$.

    Note that $f(\omega)=\theta$ for almost every $\omega\in E\setminus A$. Then,
    $$F-\int_{E\setminus A}fd\mu=0,$$
    and therefore
    $$\mathcal I=\frac{1}{\mu(E)}\left(F-\int_Afd\mu+F-\int_{E\setminus A}fd\mu\right)=\frac{1}{\mu(E)}F-\int_{A}fd\mu.$$
    Consequently,
    \begin{equation}\label{eq:xilambda}
        \mathcal I=\frac{1}{\mu(E)}B-\int_A\delta_{f(\omega)}d\mu(\omega) =B-\int_A\frac{\delta_{f(\omega)}-\delta_\theta}{d(f(\omega),\theta)}\frac{d(f(\omega),\theta)}{\mu(E)}d\mu(\omega).
    \end{equation}
    Define the measure $\lambda$ over $E$ such that $d\lambda(\omega)=d(f(\omega),\theta)d\mu(\omega)/\mu(E)$ and note that it is a probability. Indeed, since $\|\mathcal I\|_\mathcal F=1$ by the extremality of $\mathcal I$, from Proposition \ref{prop:freenormdist} it follows
    \begin{align*}
        1&=\|\mathcal I\|_\mathcal F =\frac{1}{\mu(E)}\left\|F-\int_Efd\mu\right\|_\mathcal F\\
        &=\frac{1}{\mu(E)}\int_Ed(f(\omega),\theta)d\mu(\omega)=\lambda(E).
    \end{align*}
    Therefore, by \eqref{eq:xilambda} we get
    $$\mathcal I=B-\int_A m_{f(\omega),\theta}d\lambda(\omega).$$

    Let $A_0\subseteq A$. We now distinguish three cases:
    \begin{enumerate}[(I)]
        \item If $\lambda(A_0)=0$, then $\int_{A_0}m_{f(\omega),\theta}d\lambda(\omega)=0=\int_{A_0}\mathcal I d\lambda$.
        \item If $\lambda(A_0)=1$, then
        $$\int_{A_0}m_{f(\omega),\theta}d\lambda(\omega)=\int_Am_{f(\omega),\theta}d\lambda(\omega)=\mathcal I=\int_{A_0}\mathcal I d\lambda.$$
        \item Suppose $0<\lambda(A_0)<1$. Naming $\tilde {A_0}=A\setminus A_0$ we have $0<\lambda(\tilde {A_0})<1$. Define
        $$\mathcal I_1:=\frac{1}{\lambda(A_0)}\int_{A_0}m_{f(\omega),\theta}d\lambda(\omega), \qquad \mathcal I_2:=\frac{1}{\lambda(\tilde {A_0})}\int_{\tilde {A_0}}m_{f(\omega),\theta}d\lambda(\omega),$$
        and note that $\mathcal I_1,\mathcal I_2\in B_{\mathcal F(M)}$ since $\|m_{x,y}\|_\mathcal F\le 1$ for all $x,y\in M$. Furthermore, since
        $$\mathcal I=\lambda(A_0)\mathcal I_1+(1-\lambda(\tilde {A_0}))\mathcal I_2$$
        and $\mathcal I$ is a extremal point of $B_{\mathcal F(M)}$, it follows $\mathcal I=\mathcal I_1$. Therefore, $$\int_{A_0}m_{f(\omega),\theta}d\lambda(\omega)=\lambda(A_0)\mathcal I=\int_{A_0}\mathcal I d\lambda$$
    \end{enumerate}
    From all three cases we deduce $\int_{A_0}(m_{f(\omega),\theta}-\mathcal I) d\lambda(\omega)=0$ for all $A_0\subseteq A$. Consequently, $m_{f(\omega),\theta}=\mathcal I$ for almost every $\omega\in A$, and therefore $f$ must be constant in $A$ $\mu$-a.e. Let $p\in M$ be this constant value.

    Finally, observe that
    $$1=\|\mathcal I\|_\mathcal F=\mu(A)\|\delta_p\|_\mathcal F=\mu(A)d(p,\theta)$$
\end{proof}

Note that the condition for a free integral to be an extreme point of $B_{\mathcal{F}(M)}$ is restrictive. A natural question is then to determine where the free integral lies within $\mathcal{F}(M)$. The following result provides an answer: the normalised free integral over any measurable set of positive measure belongs to the closed convex hull of the evaluation functionals along the integrand.
\begin{theorem}\label{thm:convexint}
    Let $f\colon\Omega\to M$ be a free integrable function. For any measurable set $E\in\Sigma$ such that $\mu(E)>0$ it holds
    $$\frac{1}{\mu(E)}\int_Efd\mu\in\overline{\mathrm{conv}}\left\{\delta_{f(\omega)}\colon w\in E\right\}$$
    for any $\mu$-a.e. representative $f$.
\end{theorem}
\begin{proof}
    Name $K:=\overline{\mathrm{conv}}\left\{\delta_{f(\omega)}\colon \omega\in E\right\}$, and name $U:=\{\xi\}$ where $$\xi:=\frac{1}{\mu(E)}\int_Efd\mu.$$
    Suppose that $\xi\notin K$, and hence $U\cap K=\emptyset$. Note that $K$ is closed and convex while $U$ is convex and compact. By the Hahn-Banach Theorem \cite[Theorem 1.7]{brezisfunctional}, there exist a closed hyperplane $H$ in $\mathcal F(M)$ that strictly separates $K$ and $U$. That is, there exist $\alpha\in\mathbb R$ and $T\in\mathcal F(M)^*$ such that $$H=\{m\in\mathcal F(M)\colon \langle m, T\rangle=\alpha\},$$
    and
\begin{equation}\label{eq:separation}
        \langle \xi,T \rangle>\alpha, \quad \langle\eta,T\rangle\le\alpha \ \ \text{for all} \ \eta\in K.
    \end{equation}
    Since $\mathcal F(M)^*\cong\mathrm{Lip}_0(M)$, there exist $\varphi\in\mathrm{Lip}_0(M)$ such that $T\circ \delta=\varphi$.
    
    For each $\omega\in\Omega$, note that $\delta_{f(\omega)}\in K$. Consequently, $$\varphi\circ f(\omega)=\langle\delta_{f(\omega)}, T\rangle\le\alpha.$$ Integrating in both sides we get $B-\int_E \varphi\circ fd\mu\leq\alpha\mu(E)$.
    By Theorem \ref{thm:int_dual} we have that $T\left(F\ -\int_Efd\mu\right)\le\alpha\mu(E)$. From this we deduce $\langle\xi,T\rangle\le\alpha$, which is a contradiction with \eqref{eq:separation}. Then, we conclude $\xi\in K$ finishing the proof.
\end{proof}

\begin{proposition}\label{prop:averageextreme}
    Let $f\colon\Omega\to M$ be a free integrable function, and let $E\in\Sigma$ such that $\mu(E)>0$. If
    $$\xi:=\frac{1}{\mu(E)}\int_Efd\mu$$
    is a extreme point of $K:=\overline{\operatorname{conv}}\{\delta_{f(\omega)}\colon\omega\in E\}$, then $f$ is constant in $E$ $\mu$-a.e.
\end{proposition}
\begin{proof}
    We first claim that if for all measurable $C\subseteq E$ with $\mu(C)>0$ it holds
    \begin{equation}\label{eq:extconstant}
        \xi=\frac{1}{\mu(C)}\int_Cfd\mu,
    \end{equation}
    then $f$ is constant $\mu$-a.e. Indeed, for any such a $C$, by \eqref{eq:extconstant} we get $0=B-\int_C(\delta_f-\xi)d\mu$. Therefore, $\delta_f=\xi$ $\mu$-a.e. Since $\delta$ is an isometry and thus injective it follows that $f$ is constant $\mu$-a.e.

    Suppose that $f$ is not constant $\mu$-a.e. in $E$. In this case, there exist $C \in\Sigma$ subset of $E$ with $\mu(C)>0$ such that
    $$\xi_C:=\frac{1}{\mu(C)}\int_Cfd\mu\neq\xi.$$

    If $\mu(E\setminus C)>0$ we can write
    $$\xi=\frac{\mu(C)}{\mu(E)}\xi_C+\frac{\mu(E\setminus C)}{\mu(E)}\xi_{E\setminus C}.$$
    By Theorem \ref{thm:convexint} it holds $\xi_C,\xi_{E\setminus C}\in K$.
    Since $\xi$ is an extreme point it follows $\xi_C=\xi$, which is a contradiction.

    On the other hand, suppose that $\mu(E\setminus C)=0$. Then, $\mu(C)=\mu(E)$ and
    $$\xi_C:=\frac{1}{\mu(C)}\int_Cfd\mu=\frac{1}{\mu(E)}\int_Efd\mu=\xi.$$
    Since this is also a contradiction, we conclude that $f$ is constant in $E$ $\mu$-a.e.
\end{proof}

Finally, although every free integral is a convex integral of molecules,
the converse is generally false. The following result shows that, even
in the simplest nontrivial setting, there exist elementary molecules
that cannot arise as free integrals.

\begin{proposition}\label{prop:notfreeint}
    Let $(M,d,\theta)$ be a pointed metric space containing at least
    two points different from $\theta$. Then, there exist convex integrals
    of molecules in $\mathcal{F}(M)$ that are not the free integral of any $f\in L^1_{\mathcal{F}}(\mu,M)$ over any $E\in\Sigma$.
\end{proposition}
\begin{proof}
    Let $x,y\in M$ with $x\neq y$ and $x,y\neq\theta$. We claim that the
    elementary molecule $m_{xy}$ is not a free integral. Note that $m_{xy}$
    is trivially a convex integral of molecules, since it corresponds to the
    discrete de Leeuw representation $\delta_{(x,y)}$.

    Suppose, for the sake of contradiction, that there exist a measurable set $E\in\Sigma$ with $\mu(E)>0$, and a function $f\in L^1_{\mathcal{F}}(\mu,M)$
    such that
    \begin{equation}\label{eq:mxyfreeint}
        m_{xy} = F-\int_E f\,d\mu.
    \end{equation}
    Consider the function
    $\tilde{\varphi}(\cdot):=d(\theta,\cdot)\in\mathrm{Lip}_0(M)$,
    which satisfies $\|\tilde\varphi\|_{\mathrm{Lip}}=1$. On the one
    hand, by Theorem \ref{cor:duality_free_integral} applied to $\tilde\varphi$ and
    equation~\eqref{eq:mxyfreeint}, we obtain
    \begin{equation}\label{eq:freeintnn}
        \left\langle m_{xy},\,\tilde{\varphi}\right\rangle
        = \left\langle F\text{-}\int_E f\,d\mu,\,\tilde{\varphi}
        \right\rangle
        = \int_E \tilde{\varphi}\circ f\,d\mu
        = \int_E d(\theta,f)\,d\mu \geq 0,
    \end{equation}
    since $d(\theta,f(\omega))\geq 0$ for all $\omega\in\Omega$ and
    $\mu$ is a positive measure. On the other hand, a direct computation
    gives
    \begin{equation}\label{eq:mxysign}
        \langle m_{xy},\,\tilde{\varphi}\rangle
        = \frac{d(\theta,x)-d(\theta,y)}{d(x,y)}.
    \end{equation}
    We distinguish two cases.
    \begin{enumerate}[(i)]
        \item Suppose $d(\theta,x)\neq d(\theta,y)$. Without loss of
        generality, assume $d(\theta,x)<d(\theta,y)$ (otherwise, the same argument can be applied to $m_{yx}$).
        Then \eqref{eq:mxysign} yields
        $\langle m_{xy},\tilde{\varphi}\rangle<0$, which contradicts
        \eqref{eq:freeintnn}.

        \item Suppose $d(\theta,x)=d(\theta,y)$. Then
        $\langle m_{xy},\tilde\varphi\rangle=0$ by \eqref{eq:mxysign}.
        However, combining \eqref{eq:mxyfreeint} with
        Proposition \ref{prop:freenormdist} yields
        \[
            1 = \|m_{xy}\|_{\mathcal{F}}
            = \left\|F\text{-}\int_E f\,d\mu\right\|_{\mathcal{F}}
            = \int_E d(\theta,f)\,d\mu,
        \]
        Note that the right-hand side equals
        $\langle m_{xy},\tilde\varphi\rangle$ by
        \eqref{eq:freeintnn}, which is a contradiction.
    \end{enumerate}
    Since both cases lead to a contradiction, $m_{xy}$ cannot be
    represented as a free integral.
\end{proof}

\begin{remark}\label{rem:positivity}
    The results above reveal the following constraint on the free
    integral: it always produces positive elements of
    $\mathcal{F}(M)$. Recall that an element $m\in\mathcal{F}(M)$ is said to be
    positive if $\langle m,\varphi\rangle\geq 0$ for every
    $\varphi\in\mathrm{Lip}_0(M)$ satisfying $\varphi\geq 0$
    pointwise. Indeed, for any $f\in L^1_{\mathcal{F}}(\mu,M)$, any
    $E\in\Sigma$, and any non-negative
    $\varphi\in\mathrm{Lip}_0(M)$, Corollary \ref{cor:duality_free_integral} gives
    \[
        \left\langle F\text{-}\int_E f\,d\mu,\,\varphi\right\rangle
        = \int_E \varphi\circ f\,d\mu \geq 0,
    \]
    since $\varphi(f(\omega))\geq 0$ for all $\omega\in\Omega$ and
    $\mu$ is a non-negative measure. This non-negativity is intrinsic to the
    construction: the coefficients $\mu(A_i\cap E)$ appearing in
    Definition \ref{def:sfintegral} are always non-negative, and this property is
    preserved in the limit defining the integral of strongly measurable
    functions.

    On the other hand, the only positive elementary molecules are those
    of the form $m_{p\theta}=\delta_p/d(p,\theta)$ for
    $p\neq\theta$. Indeed, $m_{p\theta}$ is positive since
    $\langle m_{p\theta},\varphi\rangle = \varphi(p)/d(p,\theta)\geq 0$
    for all non-negative $\varphi\in\mathrm{Lip}_0(M)$. Conversely,
    for any elementary molecule $m_{xy}$ with $x\neq y$ and $y\neq\theta$, consider the function
    $\tilde{\psi}(\cdot):=\max\{d(\theta,y)-d(\cdot,y),\,0\}$. Observe that $\tilde{\psi}\in\mathrm{Lip}_0(M)$, and it is non-negative. Furthermore, a direct computation shows $\tilde{\psi}(y)=d(\theta,y)>0$, and $\tilde\psi(x) < d(\theta,y)$. Hence,
    \[
        \langle m_{xy},\tilde{\psi}\rangle
        = \frac{\tilde{\psi}(x)-d(\theta,y)}{d(x,y)} < 0.
    \]
    Consequently, the free integral can only reproduce elementary molecules of the
    form $m_{p\theta}$.
\end{remark}

\section{A detailed example}\label{sec:example}
Examples of genuine metric spaces that cannot be endowed with a linear structure are provided by metric trees. A finite metric tree is a finite metric space $M=\{\theta,x_1,\ldots,x_n\}$ such that
for every pair $a,b\in M$ there exists a unique sequence of distinct points $a=y_0,y_1,\ldots,y_p=b$ in $M$ satisfying
\[
    d(a,b)=\sum_{k=0}^{p-1}d(y_k,y_{k+1}).
\]
Such a sequence is called the geodesic from $a$ to $b$. The set of edges $\mathcal{E}\subset M\times M$ consists of
all pairs $(a,b)$ such that the geodesic from $a$ to $b$ contains no intermediate points. A sequence of edges is called a path.

For example, consider $M_0=\{1,2,3\}\times\{1,2,3\}$ equipped with the river metric. In this metric, the first column
$\{(i,1)\colon i=1,2,3\}$ acts as a ``river": any path between
points in different rows must pass through it. Formally,
\[
    d\bigl((i,j),(i',j')\bigr)=
    \begin{cases}
        (j-1)+|i-i'|+(j'-1) & \text{if } i\neq i',\\
        |j-j'| & \text{if } i=i'.
    \end{cases}
\]
Neither the Euclidean metric nor the Manhattan distance on this grid produce a metric tree, but the river metric does: the unique geodesic between points in different rows is forced through the river, preventing cycles. We take $\theta=(3,1)$ as the base point and denote this pointed metric space by $(M,d,\theta)$. Once a base point $\theta$ is fixed, the metric tree can be naturally viewed as a directed graph by assigning an orientation to its edges. In particular, we adopt the convention of orienting edges toward $\theta$. In Figure \ref{fig:metrictreegrapgh} we provide an ilustrative representation of the metric tree $M_0$.

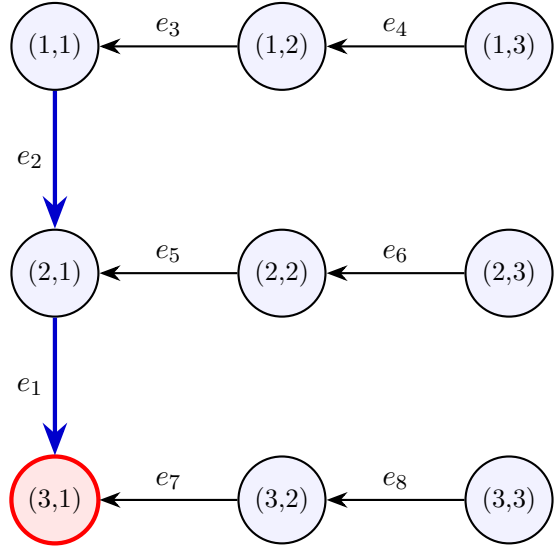
\begin{figure}
    \centering
\begin{tikzpicture}[
    >={Stealth[scale=1.2]}, 
    node distance=2.5cm, 
    thick,
    vertex/.style={circle, draw, fill=blue!5, minimum size=1cm, font=\small},
    sink/.style={circle, draw=red, fill=red!10, ultra thick, minimum size=1cm, font=\small}
]
    \foreach \i in {1,2,3} {
        \foreach \j in {1,2,3} {
            \node[vertex] (n\i\j) at (\j*3, -\i*3) {(\i,\j)};
        }
    }

    \node[sink] at (n31) {(3,1)};

    \draw[->, ultra thick, blue!80!black] (n11) -- node[left, black] {$e_2$} (n21);
    \draw[->, ultra thick, blue!80!black] (n21) -- node[left, black] {$e_1$} (n31);

    \draw[->] (n13) -- node[above] {$e_4$} (n12);
    \draw[->] (n12) -- node[above] {$e_3$} (n11);

    \draw[->] (n23) -- node[above] {$e_6$} (n22);
    \draw[->] (n22) -- node[above] {$e_5$} (n21);

    \draw[->] (n33) -- node[above] {$e_8$} (n32);
    \draw[->] (n32) -- node[above] {$e_7$} (n31);
\end{tikzpicture}
    \caption{Representation of the directed graph provided by the metric tree $M_0$. The direction of the edges is chosen to fix a sign convetion when identifying with $l_1^8$.}
    \label{fig:metrictreegrapgh}
\end{figure}

We also endow $\Omega = M_0$ with the uniform probability measure on two points:
\[
    \mu(\{(i,j)\})=
    \begin{cases}
        \tfrac{1}{2} & \text{if }
        (i,j)\in\{(1,2),(2,3)\},\\
        0 & \text{otherwise},
    \end{cases}
\]
and denote the resulting probability space by
$(\Omega,2^{M_0},\mu)$. 

Consider the functions
$f,g\colon\Omega\to M$ given by $f(\omega)=\omega$ (the identity)
and $g$ defined as $g(i,j)=(j,4-i)$. The images of the points with non-zero probability are:
\[
    f(1,2)=(1,2),\quad f(2,3)=(2,3),\qquad
    g(1,2)=(2,3),\quad g(2,3)=(3,2).
\]

The Fréchet integral of $f$ is the minimiser of
\[
    \Phi_f(y):=\int_\Omega d^2(f(\omega),y)\,d\mu(\omega)
    =\tfrac{1}{2}\,d^2((1,2),y)+\tfrac{1}{2}\,d^2((2,3),y)
\]
over $y\in M$. Evaluating at each point of $M$ we get
\[
\renewcommand{\arraystretch}{1.1}
\begin{array}{c|ccccccccc}
    y & \scriptstyle(1,1) & \scriptstyle(1,2)
    & \scriptstyle(1,3) & \scriptstyle(2,1)
    & \scriptstyle(2,2) & \scriptstyle(2,3)
    & \scriptstyle(3,1) & \scriptstyle(3,2)
    & \scriptstyle(3,3) \\ \hline
    \Phi_f(y) & 5 & 8 & 13 & 4 & 5 & 8 & 9 & 16 & 25
\end{array}
\]
The minimum is attained uniquely at $y=(2,1)$ with
$\Phi_f((2,1))=4$. Similarly, for $g$ we compute
\[
    \Phi_g(y)
    =\tfrac{1}{2}\,d^2((2,3),y)+\tfrac{1}{2}\,d^2((3,2),y),
\]
which gives
\[
\renewcommand{\arraystretch}{1.1}
\begin{array}{c|ccccccccc}
    y & \scriptstyle(1,1) & \scriptstyle(1,2)
    & \scriptstyle(1,3) & \scriptstyle(2,1)
    & \scriptstyle(2,2) & \scriptstyle(2,3)
    & \scriptstyle(3,1) & \scriptstyle(3,2)
    & \scriptstyle(3,3) \\ \hline
    \Phi_g(y) & 9 & 16 & 25 & 4 & 5 & 8 & 5 & 8 & 13
\end{array}
\]
Again, the minimum is attained uniquely at $y=(2,1)$ with
$\Phi_g((2,1))=4$. Therefore, the Fréchet integral of $f$ and $g$ coincide, both equal to $(2,1)$, with the same optimal cost.

On the other hand, by Definition \ref{def:sfintegral},
\[
    F-\int_\Omega f\,d\mu
    = \tfrac{1}{2}\,\delta_{(1,2)}
    +\tfrac{1}{2}\,\delta_{(2,3)},
    \qquad
    F-\int_\Omega g\,d\mu
    = \tfrac{1}{2}\,\delta_{(2,3)}
    +\tfrac{1}{2}\,\delta_{(3,2)}.
\]
These are distinct elements of $\mathcal{F}(M)$. Indeed, consider $\varphi(i,j):=i-3\in\mathrm{Lip}_0(M)$, which
satisfies $\varphi(\theta)=\varphi(3,1)=0$ and
$\|\varphi\|_{\mathrm{Lip}}=1$. By
Corollary~\ref{cor:duality_free_integral},
\[
    \left\langle F-\int_\Omega f\,d\mu,\,
    \varphi\right\rangle
    = \tfrac{1}{2}(1-3)+\tfrac{1}{2}(2-3)
    = -\tfrac{3}{2},
\]
while
\[
    \left\langle F-\int_\Omega g\,d\mu,\,
    \varphi\right\rangle
    = \tfrac{1}{2}(2-3)+\tfrac{1}{2}(3-3)
    = -\tfrac{1}{2}.
\]
Moreover, the norms of the free integrals also differ. By
Proposition \ref{prop:freenormdist},
\[
    \left\|F-\int_\Omega f\,d\mu\right\|_{\mathcal{F}}
    = \tfrac{1}{2}\,d(\theta,(1,2))+\tfrac{1}{2}\,d(\theta,(2,3))
    = \tfrac{1}{2}\cdot 3+\tfrac{1}{2}\cdot 3=3,
\]
while
\[
    \left\|F-\int_\Omega g\,d\mu\right\|_{\mathcal{F}}
    = \tfrac{1}{2}\,d(\theta,(2,3))+\tfrac{1}{2}\,d(\theta,(3,2))
    = \tfrac{1}{2}\cdot 3+\tfrac{1}{2}\cdot 1=2.
\]
This example illustrates that, while the Fréchet integral cannot distinguish between two distinct functions, the free integral captures their distributional differences as
distinct elements of $\mathcal{F}(M)$.

To further explore this example, we provide a description of the free space of our metric space $M_0$. For this, we refer to \cite{godard2010tree,weaverlipschitz}. If $T$ is a finite metric tree, then $\mathcal F(T)$ is linearly isometric to $l_1^N$ for any $\theta\in T$, where $N=\operatorname{card}(T)-1$. Moreover, each edge $e_i$ can be identifyed with an element $b_i$ of the canonical basis of $l_1^N$. Furthermore, if the directed edge $e_i$ connects $u$ to a subsequent node $v$ (closer to $\theta$), we establish the following identification:
$$b_i := \delta_u - \delta_v.$$
Since $\delta_\theta = 0_{\mathcal{F}(M_0)}$, we can systematically deduce the evaluation functional $\delta_x$ for any point $x \in T$ going through the unique geodesic from $x$ to $\theta$. 

In our case, we can identify $\mathcal F(M_0)$ with $l_1^8$. We can identify each molecule $\delta_{(i,j)}$ with a vector of $l_1^8$ in the following way:
\begin{itemize}
    \item From $e_1$, we have $b_1 = \delta_{(2,1)} - \delta_{(3,1)}=\delta_{(2,1)}$, hence $\delta_{(2,1)} = b_1$.
    \item From $e_2$, we have $b_2 = \delta_{(1,1)} - \delta_{(2,1)}$, hence $\delta_{(1,1)} = b_1 + b_2$.
    \item From $e_3$, we have $b_3 = \delta_{(1,2)} - \delta_{(1,1)}$ hence $\delta_{(1,2)} = b_1 + b_2 + b_3$.
\end{itemize}

Consequently,
$$F-\int_\Omega fd\mu=\frac{1}{2}(b_1+b_2+b_3)+\frac{1}{2}(b_1+b_5+b_6)=(1,\tfrac{1}{2},\tfrac{1}{2},0,\tfrac{1}{2},\tfrac{1}{2},0,0),$$
and
$$F-\int_\Omega gd\mu=\frac{1}{2}(b_1+b_5+b_6)+\frac{1}{2}(b_7)=(\tfrac{1}{2},0,0,0,\tfrac{1}{2},\tfrac{1}{2},\tfrac{1}{2},0).$$

This ability to capture spatial distributions makes the free integral a powerful tool for applied fields. In image recognition, standard approaches usually model digital images as elements of a linear space. For example, a grayscale image in PNG format can be modeled as a matrix in $\{0,1,\ldots,255\}^{H \times W},$ where each matrix entry represents a pixel intensity, and $H$ and $W$ denote the image height and width, respectively. However, the induced metrics for these vector spaces are often inadequate for image recognition tasks. These metrics, such as the euclidean distance, rely on pixel-to-pixel differences, failing to capture global overall similarities. To address it, we will identify images as probability measures. By mapping these probability measures into $l_1^N$ via the free integral, we obtain a linear representation that intrinsically respects the spatial geometry of the pixels, overcoming the limitations of pixel-to-pixel metrics.

Instead of relying on geometric constructions, our approach formulates the integration process within the framework of Lipschitz-free spaces, where we introduce the free integral via a procedure that is analogous to the classical theory of Bochner integration.

\section*{Statements and Declarations}
This research was funded by the Agencia Estatal de Investigación, grant number PID2022-138342NB-I00. The second author was supported by a contract of the Programa de Ayudas de Investigación y Desarrollo (PAID-01-24), Universitat Politècnica de València.

\printbibliography

\end{document}